\documentclass[preprint,11pt]{elsarticle}

\usepackage{ifpdf}
\usepackage{amsmath}
\usepackage{amsfonts}
\usepackage{amsthm}
\usepackage{mathrsfs}
\usepackage{color}
\usepackage{graphicx}
\usepackage{subfigure}
\usepackage{float}
\usepackage{overpic}
\usepackage{threeparttable}
\usepackage{epstopdf}
\usepackage{dcolumn}
\usepackage{multirow}
\usepackage{booktabs}
\biboptions{numbers,sort&compress}

\usepackage{graphicx,natbib,amssymb,lineno}
\ifpdf
\usepackage[%
  pdftitle={Instructions for use of the document class
    elsart},%
  pdfauthor={},%
  pdfsubject={The preprint document class elsart},%
  pdfkeywords={instructions for use, elsart, document class},%
  pdfstartview=FitH,%
  bookmarks=true,%
  bookmarksopen=true,%
  breaklinks=true,%
  colorlinks=true,%
  linkcolor=blue,anchorcolor=blue,%
  citecolor=blue,filecolor=blue,%
  menucolor=blue,pagecolor=blue,%
  urlcolor=blue]{hyperref}
\else
\usepackage[%
  breaklinks=true,%
  colorlinks=true,%
  linkcolor=blue,anchorcolor=blue,%
  citecolor=blue,filecolor=blue,%
  menucolor=blue,pagecolor=blue,%
  urlcolor=blue]{hyperref}
\fi

\makeatletter
\def\elsartstyle{%
    \def\normalsize{\@setfontsize\normalsize\@xiipt{14.5}}
    \def\small{\@setfontsize\small\@xipt{13.6}}
    \let\footnotesize=\small
    \def\large{\@setfontsize\large\@xivpt{18}}
    \def\Large{\@setfontsize\Large\@xviipt{22}}
    \skip\@mpfootins = 18\p@ \@plus 2\p@
    \normalsize
} \@ifundefined{square}{}{} \makeatother

\newtheorem{theorem}{Theorem}[section]
\newtheorem{lemma}[theorem]{Lemma}
\newtheorem{proposition}[theorem]{Proposition}
\newtheorem{corollary}[theorem]{Corollary}
\newtheorem{conjecture}[theorem]{Conjecture}

\theoremstyle{definition}

\newtheorem{remark}[theorem]{Remark}

\makeatletter
\def\ps@pprintTitle{%
  \let\@oddhead\@empty
  \let\@evenhead\@empty
  \let\@oddfoot\@empty
  \let\@evenfoot\@oddfoot
}
\makeatother

\def\ps@pprintTitle{%
  \let\@oddhead\@empty
  \let\@evenhead\@empty
  \def\@oddfoot{\reset@font\hfil\thepage\hfil}
  \let\@evenfoot\@oddfoot
}

\begin{document}

\begin{frontmatter}

\title{
A codimension-$5$ Hopf bifurcation yielding $5$ limit cycles \\ 
in a SIRS epidemic model with nonlinear incidence rate
}

\author[myaddress1]
{Pei Yu\corref{mycorrespondingauthor}}
\cortext[mycorrespondingauthor]{Corresponding author}
\ead{pyu@uwo.ca}
\author[myaddress1,myaddress2]
{Wanyue Tang}
\ead{wtang294@uwo.ca}
\author[myaddress3]
{Yanni Zeng}
\ead{yazeng@ttu.edu}
\address[myaddress1]{Department of Mathematics,
Western University, London, Ontario, N6A~5B7, Canada \vspace*{0.05in}}
\address[myaddress2]{School of Mathematics and Statistics, Lanzhou University,
Lanzhou, Gansu 730000, China \vspace*{0.05in}}
\address[myaddress3]{Mathematics and Statistics, Texas Tech University,
Lubbock, TX 79409, USA}


\begin{abstract}
In a recent paper published in the Journal of Differential Equations 
(384, 2024), Cui and Zhao investigated an SIRS epidemic model with the 
nonlinear incidence rate $\frac{kI^p}{1+\alpha I^q}$, where $p>0$ and 
$q\geq0$ are arbitrary real numbers. They proved that the codimension of 
a Bogdanov-Takens bifurcation in this model is at most two, while the 
codimension of the Hopf bifurcation was left as an open problem. 
In this work, we investigate this problem while keeping $p$ and $q$ free 
throughout the analysis. A strategic nondimensionalization and 
parametrization remove the exponential dependence of the positive 
equilibrium and convert the generalized-Hopf conditions into algebraic 
focus equations coupled to explicit semialgebraic admissibility regions. 
Using algebraic focus reduction, strict semialgebraic admissibility, and
validated interval computation, we rigorously prove the existence of
admissible nondegenerate generalized Hopf points generating three, four,
and five small-amplitude limit cycles. In particular, we rigorously certify
a strict-interior codimension-five weak focus and a full-rank local unfolding,
thereby proving that five small-amplitude limit cycles can bifurcate from a
single positive equilibrium. We further derive exact algebraic obstructions
to codimension six and find no admissible candidate in extensive continuation
and projected-component searches. These results provide analytical and
numerical evidence for the conjecture that five is maximal, but no
codimension-six nonexistence theorem is claimed. The same parametrization
also yields a simple proof that the Bogdanov-Takens codimension is at most two.
\end{abstract}


\end{frontmatter}

\section{Introduction and main results}\label{Sec-1}  

Mathematical models have long been an effective tool for understanding 
the transmission and long-term dynamics of infectious diseases, and epidemic 
modeling has been extensively studied over the past several decades; see, 
for example,~\cite{diekmann2000mathematical,hethcote2000mathematics,
kermack1927contribution,anderson1979population,anderson1991infectious}. 
Among the classical compartmental models, the 
susceptible-infective-recovered-susceptible (SIRS) model is particularly 
suitable for describing diseases in which recovered individuals may 
lose their immunity and become susceptible again; see, 
for example,~\cite{hethcote1976qualitative,hethcote2000mathematics,
anderson1991infectious}.

Let $S(t)$, $I(t)$, and $R(t)$ denote the
numbers of susceptible, infective, and recovered individuals at time $t$,
respectively. Then, a standard SIRS model is given by
\begin{equation}\label{Eqn1} 
\left\{
\begin{array}{ll} 
\dfrac{d S}{d t}= \Lambda-dS-f(I)S+\delta R,\\[2.0ex] 
\dfrac{d I}{d t}=f(I)S-(d+\mu)I,\\[2.0ex] 
\dfrac{d R}{d t} =\mu I-(d+\delta)R,
\end{array} 
\right. 
\end{equation}
where the parameters $\Lambda, d, \mu, \delta$ all take positive values. 
Here, new individuals enter the susceptible population at recruitment rate 
$\Lambda$; infected individuals recover at rate $\mu$; and recovered 
individuals lose immunity and return to the susceptible class at 
rate $\delta$. The parameter $d$ denotes the natural death rate, 
and the term $f(I)S$ represents the incidence rate, which determines 
the rate of new infections.

The choice of the incidence function plays a crucial role in determining 
the qualitative dynamics of disease transmission.
In the pioneering work of 
Kermack and McKendrick~\cite{kermack1927contribution}, 
the classical SIR model considers the bilinear incidence rate $kIS$, 
assuming that the infection rate increases linearly with the number of 
infective individuals. This assumption, however, may become unrealistic 
when the prevalence of infection is high, as susceptible individuals 
may reduce their social contacts or adopt protective measures in 
response to the increasing risk of infection. Such behavioral and 
psychological responses can inhibit disease transmission and motivate 
the use of nonlinear incidence functions. To incorporate such behavioral 
and psychological effects into epidemic models, Capasso and 
Serio~\cite{capasso1978generalization} introduced the 
saturated incidence function
$$ 
f(I)S=\frac{kIS}{1+\alpha I},
$$
where the factor $(1+\alpha I)^{-1}$ represents the inhibition of 
disease transmission associated with behavioral or psychological 
responses to increasing disease prevalence.

A more general nonlinear incidence function was subsequently introduced
by Liu, Levin, and Iwasa~\cite{liu1986influence} in the form
\begin{equation}\label{Eqn2} 
f(I)S=\frac{kI^pS}{1+\alpha I^q},
\end{equation}
where $k>0$, $p>0$, $\alpha \geq 0$, and $q\geq0$. Depending on the relative values of $p$ and $q$, the incidence function
is unbounded when $p>q$ and saturated when $p=q$; see~\cite{tang2008coexistence}.
When $p<q$, the incidence function is nonmonotone: it initially increases
but eventually decreases as the number of infected individuals grows,
reflecting the suppression of transmission by behavioral or protective
responses among susceptible individuals. 
Formula~\eqref{Eqn2} includes several commonly used incidence 
functions as special cases. For example, the bilinear incidence $kIS$ 
is recovered when $\alpha=0$ and $p=1$, whereas the saturated incidence
$\frac{kIS}{1+\alpha I}$ corresponds to $p=q=1$. 
Thus, allowing $p$ and $q$ to vary provides a flexible framework for
capturing a wide range of nonlinear transmission mechanisms.

With the incidence function~\eqref{Eqn2}, the SIRS model~\eqref{Eqn1} becomes
\begin{equation}\label{Eqn3} 
\left\{
\begin{array}{ll}
\displaystyle
\dfrac{d S}{d t}= \Lambda-dS-\frac{kI^pS}{1+\alpha I^q}+\delta R,\\[2.0ex]
\dfrac{d I}{d t}=\dfrac{kI^pS}{1+\alpha I^q}-(d+\mu)I,\\[2.0ex] 
\dfrac{d R}{d t}=\mu I-(d+\delta)R.
\end{array} 
\right. 
\end{equation}
The dynamics of this model have been investigated extensively. 
Hu et al.~\cite{hu2011bifurcations} studied the existence and stability 
of equilibria and
showed that system~\eqref{Eqn3} has no limit cycles when $0<p\leq1$. For
$p>1$, they demonstrated that the system may undergo several types of
local bifurcations, including Hopf and Bogdanov--Takens (BT) bifurcations.
However, the exact codimension of the Hopf bifurcation was left
undetermined.

Several subsequent studies have investigated system~\eqref{Eqn3} for
particular choices of the exponents $p$ and $q$. For $p=q=2$, Ruan and
Wang~\cite{ruan2003dynamical} and Tang et al.~\cite{tang2008coexistence} 
showed that the system can
undergo a BT bifurcation of codimension two and a
degenerate Hopf bifurcation of codimension two. For $p=1$ and $q=2$,
Xiao and Ruan~\cite{xiao2007global} performed a global analysis and 
showed that the system has no nontrivial periodic orbits. Under the condition
$p=q$, Zhang et al.~\cite{zhang2022bifurcations} identified several types 
of bifurcations, including saddle-node bifurcations, codimension-two 
BT bifurcations, and codimension-two degenerate Hopf 
bifurcations. Other results on SIRS epidemic models with various types 
of incidence functions can be found 
in~\cite{liu1986influence,lizana1996multiparametric,lu2019bifurcation,
xiao2006qualitative,zhou2007bifurcations} and the references therein. 
In particular, Lu et al.~\cite{lu2021global} showed that an SIRS 
model with the generalized nonmonotone incidence rate
$
\frac{kIS}{1+\beta I+\alpha I^2}
$
can undergo a codimension-three Hopf bifurcation, giving rise to
three small-amplitude limit cycles.

To provide an epidemiological context for the parameter values used in the
model, we summarize in Table~\ref{table1} some representative
parameter ranges and time scales reported in the literature.  These values
should be regarded as reference values rather than universal parameter
bounds, since they depend on the disease under consideration, the population
scaling, and, for the nonlinear incidence function, the particular choices
of the incidence exponents.


\begin{table}[!h]
\centering
\caption{Representative epidemiological parameter values and ranges
\cite{el2023extending}.}
\label{table1} 
\vspace{0.15in}
\small
\begin{tabular}{|c|c|l|}
\hline
Parameter & Description & Representative value/range \\
\hline
$\Lambda$ & Recruitment rate
& $\Lambda=dN$ \\

$d$ & Natural mortality rate
& $3.42\times10^{-5}\ {\rm day}^{-1}$ \\

$k$ & Transmission coefficient
& Depends on $p$ and $N$ \\

$\alpha$ & Behavioral inhibition
& Depends on $q$ and $N$ \\

$\mu$ & Recovery rate
& $0.071$--$0.333\ {\rm day}^{-1}$ \\

$\delta$ & Loss-of-immunity rate
& $0.00137$--$0.00548\ {\rm day}^{-1}$ \\
\hline
\end{tabular}
\end{table}

We emphasize that the numerical values in
Table~\ref{table1} are intended only to provide
epidemiological reference scales.  In particular, the parameters $d$,
$\mu$, and $\delta$ have direct time-scale interpretations through the
average life expectancy, infectious period, and duration of immunity,
respectively.  In contrast, the numerical magnitudes of the transmission
coefficient $k$ and, especially, the inhibition parameter $\alpha$ depend
on the form and scaling of the nonlinear incidence function.
This distinction is particularly important for the generalized incidence
function considered in the present paper.  The representative values of
$\alpha$ displayed in Table~\ref{table1} were obtained for
particular integer choices of the incidence exponents and therefore should
not be interpreted as a fixed admissible range for $\alpha$. In the present model, $p$ and $q$ are allowed to be positive real numbers.
Consequently, when $p$ and $q$ differ from the integer values used in
obtaining the representative estimates in the table, the corresponding
value of $\alpha$ may have a substantially different numerical magnitude.
Here, $N=\frac{\Lambda}{d}$ denotes the total population size at demographic
equilibrium. We emphasize that model~(3) is formulated in the original,
non-normalized population variables; in particular, $I$ represents the
number of infective individuals rather than the infective proportion
$\frac{I}{N}$. This distinction is important for interpreting the parameters
in the generalized incidence function. To illustrate the role of the
inhibition parameter $\alpha$, let $I_{1/2}$ denote the number of
infective individuals at which the inhibition factor $\frac{1}{1+\alpha I^q} $
is reduced to $\frac{1}{2}$. Then
\[
\alpha I_{1/2}^q=1,
\qquad
\alpha=(I_{1/2})^{-q}.
\]
Thus, $\alpha$ may be interpreted in terms of the half-inhibition
threshold $I_{1/2}$. Since the model is formulated in the original,
non-normalized population variables, the numerical magnitude and
physical dimension of $\alpha$ depend on the exponent $q$ and the
population scale. Therefore, numerical values of $\alpha$
corresponding to different values of $q$ or different population
scales should not be compared directly.

The transmission coefficient $k$ can be interpreted similarly in terms
of the population scale. In the classical special case $p=1$ and
$\alpha=0$, one has
\[
R_0=\frac{kN}{d+\mu},
\qquad
k=\frac{R_0(d+\mu)}{N}.
\]
Thus, taking $R_0$ between $1$ and $7$ and using the values of $\mu$
listed in Table~\ref{table1} gives
\[
k\approx(0.071\text{--}2.33)N^{-1}\ {\rm day}^{-1}.
\]
This again illustrates the importance of the population scale in the
non-normalized formulation. For general $p>0$, however, the physical
dimension and population scaling of $k$ depend on $p$, and therefore
numerical values of $k$ corresponding to different incidence exponents
should not be compared without taking this scaling into account.

Thus, in interpreting the dimensional parameter examples presented below,
greater significance should be attached to the epidemiologically meaningful
time scales associated with $d$, $\mu$, and $\delta$ than to a direct
numerical comparison of $\alpha$ with values obtained for different choices
of $p$ and $q$.

These results indicate that nonlinear incidence functions can generate
rich and complicated dynamics in SIRS epidemic models. In particular,
the occurrence of Hopf bifurcations suggests that the model can exhibit
persistent epidemic oscillations. A natural and fundamental question is
therefore how many small-amplitude limit cycles can be generated 
by a single Hopf bifurcation when the exponents in the nonlinear 
incidence function are allowed to vary over the positive real numbers. 
This question was recently considered by Cui and Zhao~\cite{cui2024saddle}
for the general model~\eqref{Eqn3}. They proved that the codimension of a
BT bifurcation is at most two, while the codimension of
the Hopf bifurcation remained open. In particular, their results did
not determine the maximal degeneracy of the Hopf bifurcation or the
number of small-amplitude limit cycles that can bifurcate from a single 
Hopf critical point. 

The purpose of this paper is to resolve this open problem in the sense of
establishing a rigorous lower bound for the Hopf codimension. We prove that
the positive equilibrium of system~\eqref{Eqn3} admits an admissible
generalized Hopf bifurcation of codimension at least five and, more
precisely, that a single Hopf point can generate at least five
small-amplitude limit cycles. The three-, four-, and five-cycle results are
proved rigorously: the relevant weak foci are isolated by algebraic reduction
and validated interval computation, strict admissibility is certified, and
the required unfolding nondegeneracy is verified. By contrast, our
codimension-six analysis gives analytical and numerical evidence for
nonexistence but is stated only as a conjecture. Thus, the rigorous conclusion
of the present paper is that the Hopf codimension is at least five, while the
possible maximality of five remains open.

To facilitate the bifurcation analysis, we first reduce 
system~\eqref{Eqn3} to a planar system. Let $N(t)=S(t)+I(t)+R(t)$ 
denote the total population size. It follows directly 
from~\eqref{Eqn3} that
$$
\frac{dN}{dt}=\frac{d}{dt}(S+I+R)=\Lambda-d(S+I+R).
$$
Consequently,
\[
S(t)+I(t)+R(t) \longrightarrow \frac{\Lambda}{d}
\qquad \text{as } t\to\infty,
\]
and the plane $S+I+R=\frac{\Lambda}{d}$ is invariant. Hence, the
limit sets of system~\eqref{Eqn3} are contained in this plane.
Eliminating $S$ by
$
S=\frac{\Lambda}{d}-I-R,
$
we obtain the planar system
\begin{equation}\label{Eqn4} 
\left\{
\begin{array}{ll}
\dfrac{d I}{dt} = \dfrac{kI^p}{1+\alpha I^q}
\left(\dfrac{\Lambda}{d}-I-R\right) -(d+\mu)I,\\[2.0ex] 
\dfrac{d R}{dt} =\mu I-(d+\delta)R.
\end{array}
\right. 
\end{equation}
The corresponding positively invariant region is
\begin{equation}\label{Eqn5} 
\Omega = \left\{
(I,R): I\geq0,\ R\geq0,\ I+R\leq\frac{\Lambda}{d}
\right\}.
\end{equation}
Based on the planar system~\eqref{Eqn4}, we analyze the successive
Lyapunov coefficients (focus values) at the Hopf point. The central result is
a rigorous certification of an admissible order-five weak focus together with
a nondegenerate local unfolding, from which five small-amplitude limit cycles
are obtained.

The main theorem is stated below; its proof is given in Section~5.

\begin{theorem}\label{Thm1.1}
For system~\eqref{Eqn4}, there exists an admissible parameter
configuration for which the positive equilibrium is a nondegenerate
order-five weak focus. Under an arbitrarily small admissible local unfolding,
at least five small-amplitude limit cycles bifurcate from this equilibrium.
Consequently, the Hopf bifurcation at the positive equilibrium has
codimension at least five.
\end{theorem} 

The rest of the paper is organized as follows. In Section 2, 
model \eqref{Eqn4} is rescaled to a dimensionless form to facilitate 
the analysis, and the positive equilibrium is normalized to $(1,1)$.
In Section 3, the stability of the equilibria is analyzed, with results 
that differ from those presented by Cui and Zhao~\cite{cui2024saddle}.
In Section 4, two special cases of Hopf bifurcation, corresponding
to $\alpha=0$ and $q=0$, are considered. Section 5 develops the focus-value
formulation directly in the parameters $(a,m,n,g,D)$ and then treats 
successively the three-, four-, five-, and six-limit-cycle problems 
under the corrected admissibility conditions. Section 6 records the 
separate $p=7$, $q=5$ comparison example. In Section 7, we provide 
a simple proof of the BT bifurcation for model \eqref{Eqn4}. 
Finally, Section 8 presents the conclusions and discusses open questions.

\section{Dimensionless model and equilibria}\label{Sec-2} 

To simplify the analysis of system \eqref{Eqn4}, 
we introduce the following scaling:
\begin{equation}\label{Eqn6} 
I = \dfrac{\Lambda}{d}\, X, \quad 
R = \dfrac{\mu \Lambda}{d(d + \delta)}\, Y, \quad 
t' = (d + \delta)\, t.  
\end{equation} 
Substituting \eqref{Eqn6} into system \eqref{Eqn4} and dropping the 
prime from $t'$ for simplicity, we obtain
\begin{equation}\label{Eqn7} 
\left\{
\begin{array}{ll} 
\dfrac{dX}{dt} = \dfrac{k_1 X^p}{1+b X^q} (1-X-sY) - r X, \\[2.0ex] 
\dfrac{dY}{dt} = X-Y,
\end{array} 
\right. 
\end{equation} 
where the new parameters are defined by  
\begin{equation}\label{Eqn8}  
k_1 = \frac{k}{d+\delta} \left(\dfrac{\Lambda}{d}\right)^p, 
\quad b = \alpha \left(\frac{\Lambda}{d}\right)^q, 
\quad s=\dfrac{\mu}{d + \delta}, \quad r = \dfrac{\mu + d}{d+\delta}. 
\end{equation}
Since $\mu,\,d,\,\delta>0$, these parameters satisfy  
\begin{equation}\label{Eqn9}  
s>0, \quad  s<r<s+1. 
\end{equation}

Suppose that system \eqref{Eqn7} has a positive (interior) equilibrium
\[
E_1=(\tilde X,\tilde X).
\]
Then $\tilde{X}$ satisfies 
\begin{equation}\label{Eqn10}  
k_1 \tilde{X}^p \big[1-(1+s)\, \tilde{X} \big]
= r \tilde{X} ( 1 + b \tilde{X}^q),  
\end{equation}
which requires 
\begin{equation}\label{Eqn11}  
0<\tilde{X} < \dfrac{1}{s+1}. 
\end{equation} 

We further introduce the scaling 
\begin{equation}\label{Eqn12}  
X = \tilde{X}\, x, \quad Y=\tilde{X}\, y. 
\end{equation} 
Substituting \eqref{Eqn12} into \eqref{Eqn7} yields 
\begin{equation}\label{Eqn13} 
\left\{
\begin{array}{ll}
\dfrac{dx}{dt} = \dfrac{K x^p}{1+B x^q} \left(\dfrac{1}{\tilde{X}}-x-sy\right) 
- r x, \\[2.0ex]     
\dfrac{dy}{dt} = x-y,
\end{array}
\right. 
\end{equation} 
where 
\begin{equation}\label{Eqn14}
K=k_1 \tilde{X}^p, \quad B=b\tilde{X}^q. 
\end{equation}

It follows from \eqref{Eqn14} that system \eqref{Eqn13} involves the 
seven quantities $p,\, q,\, K,\, B,\, \tilde X,\, s$, and $r$.
However, these quantities are subject to the equilibrium relation
\eqref{Eqn10}. Indeed, substituting $(x,y)=(1,1)$ into the first equation 
of \eqref{Eqn13} gives 
\begin{equation}\label{Eqn15}
K \big[ 1 - (1+s) \tilde{X}\big] = r \tilde{X} (1+B), 
\end{equation} 
which is equivalent to \eqref{Eqn10} in view of \eqref{Eqn14}. 

Hence, only six parameters are independent. We may therefore treat 
\begin{equation}\label{Eqn16}
p>0,\quad q\ge 0,\quad \tilde{X}>0,\quad B \ge 0,\quad s>0,\quad 
\textrm{and} \quad r \in (s,s+1),  
\end{equation} 
as six independent parameters when considering the positive 
equilibrium ${\rm E_1}=(1,1)$. 
Thus, system \eqref{Eqn13} has the positive equilibrium
\[
E_1=(1,1).
\]
In addition, it has the trivial equilibrium $ E_0=(0,0)$. 
In this paper, we focus on the Hopf bifurcation from 
the positive equilibrium ${\rm E_1}$.   

We note that the dimensionless system \eqref{Eqn13} differs from the 
dimensionless system considered in~\cite{cui2024saddle}, which is given by
\begin{equation}\label{Eqn17}
\left\{ 
\begin{array}{ll}
\dfrac{dx}{dt} = \dfrac{x^p}{1+\beta x^q} (\ell - x -y) - r x, \\[2.0ex] 
\dfrac{dy}{dt} = s x - y.
\end{array} 
\right. 
\end{equation} 

\begin{remark}\label{Rem2.1} 
At first glance, systems \eqref{Eqn13} and \eqref{Eqn17} may appear to 
be essentially equivalent. However, this apparent similarity is misleading. 
Retaining the parameter in the term $x^p$ is crucial for the subsequent 
analysis, as it eliminates exponential functions from the symbolic 
computations required to derive the focus values. This seemingly minor 
modification provides a substantial analytical advantage and significantly 
simplifies both the computation of the focus values and the solution 
of the coupled multivariate polynomial equations.
\end{remark}

As we will show in the following sections, system \eqref{Eqn13} has an 
important advantage over \eqref{Eqn17}. In particular, although system 
\eqref{Eqn17} can be used to establish the codimension of the 
BT bifurcation, it does not allow us to fully resolve 
the codimension problem for the Hopf bifurcation.

\section{Stability analysis of the equilibria}\label{Sec-3}  

Since this paper focuses on the Hopf bifurcation from 
the positive equilibrium ${\rm E_1}$, we primarily study the stability of 
${\rm E_1}$. 

Define 
\begin{equation}\label{Eqn18}
\tilde{X}_1 = 1- (s+1) \tilde{X}, \quad 
q_1 = \dfrac{(B+1)(p \tilde{X}_1-1)}{B \tilde{X}_1}, 
\quad 
p_c = 1+ \frac{1}{r} + \frac{\tilde{X}}{\tilde{X}_1}, \quad 
q_c = p-p_c, \quad 
B_{\rm H} = \dfrac{q_c}{q-q_c}, 
\end{equation}
where $0<\tilde{X}_1 <1$ follows from \eqref{Eqn11}. 
 
We first state the existence and stability conditions for ${\rm E_1}$. 

\begin{proposition}\label{Prop3.1} 
For system \eqref{Eqn13}, 
the positive equilibrium ${\rm E_1}=(1,1)$ exists if 
\eqref{Eqn10} holds together with the condition \eqref{Eqn11}. 

Moreover, ${\rm E_1}$ is locally asymptotically stable $($LAS$\,)$ 
if $s>0$ and $s<r<s+1$, and one of the following conditions holds:
\begin{enumerate}
\item[{\rm (1)}] 
$0<p \le 1$, $q\ge 0$, $0<\tilde{X}<\frac{1}{s+1}$, and $B \ge 0$.  

\vspace{0.05in} 
\item[{\rm (2)}] 
$1<p< \min\big\{p_c, \frac{1}{\tilde{X}_1} \big\}$, $q \ge 0$, 
$0<\tilde{X}<\frac{1}{s+1}$, and $B \ge 0$.

\vspace{0.05in} 
\item[{\rm (3)}] 
$\frac{1}{\tilde{X}_1}<p \le 1 + \frac{1}{r}$, $q>q_1$,  
$0<\tilde{X} \le \frac{1}{(s+1)(r+1)}$, and $B>0$.  

\vspace{0.05in} 
\item[{\rm (4)}] 
$1 + \frac{1}{r} < p < \min \big\{p_c,\frac{1}{\tilde{X}_1} \big\}$, 
$q \ge 0$, $\frac{1}{(s+1)(r+1)}<\tilde{X}< \frac{1}{s+1}$, and $B>0$.  

\vspace{0.05in} 
\item[{\rm (5)}] 
$p_c<p\le \frac{1}{\tilde{X}_1}$, $q>q_c$,  
$\frac{1}{s+1+sr}<\tilde{X}< \frac{1}{s+1}$, and $B>B_{\rm H}$.  

\vspace{0.05in} 
\item[{\rm (6)}] 
$\max\big\{\frac{1}{\tilde{X}_1},1+\frac{1}{r}\big\}<p \le p_c$, 
$q>q_1$, $0<\tilde{X}< \frac{1}{s+1+sr}$, and  $B>0$.

\vspace{0.05in} 
\item[{\rm (7)}] 
$p>\max\big\{p_c, \frac{1}{\tilde{X}_1}\big\}$, $q>\max\{q_c,q_1\}$, 
$0<\tilde{X}< \frac{1}{s+1}$, and  $B>B_{\rm H}$.
\end{enumerate} 
A Hopf bifurcation occurs from ${\rm E_1}$ at $B=B_{\rm H}$ provided 
$s>0$ and $s<r<s+1$, and either  
$$ 
p_c \!<\! p \! \le\! \tfrac{1}{\tilde{X}_1}, \ \, 
q \!>\!q_c, \ \, \tfrac{1}{s+1+sr}<\tilde{X} 
\!<\! \tfrac{1}{s+1}; \quad \textrm{or} \quad 
p\!>\! \max\left\{p_c,\tfrac{1}{\tilde{X}_1}\right\}, \ \, 
q \!>\! \max\{q_c,q_1\}, \ \, 0<\tilde{X} \!<\! \tfrac{1}{s+1}. 
$$ 
\end{proposition} 

\begin{proof} 

By \eqref{Eqn15}, we have 
\begin{equation}\label{Eqn19}
K = \dfrac{r (1+B) \tilde{X}}{\tilde{X}_1}, \\[1.0ex] 
\end{equation}  
which is positive due to $r>0$, $B \ge 0$, $\tilde{X}>0$ 
and $\tilde{X}_1>0$. 

The local stability of ${\rm E_1}$ is determined by the 
characteristic equation of the Jacobian matrix of system \eqref{Eqn13}, 
which can be written as  
\begin{equation}\label{Eqn20}
\lambda^2 - {\rm Tr}_1 \lambda + {\det}_1 =0, \\[-2.0ex] 
\end{equation} 
where 
\begin{equation}\label{Eqn21}
\begin{array}{rl} 
{\rm Tr}_1= \!\!\! & 
- \, \dfrac{1}{(1+B) \tilde{X}_1} 
\big\{B q r \tilde{X}_1+ r (1+B) \big[\big(\frac{1}{r}+1-p \big) \tilde{X}_1 
+\tilde{X} \big] \big\}  \\[2.0ex]  
= \!\!\! & - \, \dfrac{1}{(1+B) \tilde{X}_1} 
\big[B q r \tilde{X}_1+(1+B) r (p_c-p) \big] \\[2.0ex]  
= \!\!\! & -\, \dfrac{r}{1+B} (q-q_c) (B-B_{\rm H}), 
\\[2.0ex] 
\det_1 = \!\!\! & 
\dfrac{r}{(1+B) \tilde{X}_1} 
\big[B q \tilde{X}_1+(B+1) (1-p \tilde{X}_1)\big]. 
\\[3.0ex]  
\end{array} 
\end{equation}
By the Routh-Hurwitz criterion,  
${\rm E_1}$ is locally asymptotically stable if ${\rm Tr}_1<0$ and $\det_1>0$. 

First, consider $0<p\le1$. Since $\tilde{X}_1\in(0,1)$, we have
$1-p\tilde{X}_1>0$. Hence $\det_1>0$ for $q\ge0$, $r>0$ and $B \ge 0$. 
Moreover, the trace ${\rm Tr}$ is negative (see the first equation 
in \eqref{Eqn21}) under the same parameter conditions. 
Therefore, $E_1$ is LAS, which proves condition {\rm (1)}.

Next, suppose that $p>1$. From the expression for $\det_1$ in \eqref{Eqn21}, 
we obtain ${\det}_1>0$ under one of the following two sets of conditions:
\begin{equation}\label{Eqn22} 
\begin{array}{cl} 
\textrm{(i)} & 1<p \le \dfrac{1}{\tilde{X}_1}, \ \ q \ge 0,\ \ 
0<\tilde{X} < \dfrac{1}{s+1}, \ \ B \ge 0; 
\\[2.0ex] 
\textrm{(ii)} & 
p > \dfrac{1}{\tilde{X}_1}, \ \ q>q_1, \ \ 
0<\tilde{X} < \dfrac{1}{s+1}, \ \ B > 0.
\end{array} 
\end{equation} 
On the other hand, it follows from the first three equations 
\eqref{Eqn21} that ${\rm Tr}_1<0 $ requires one of the following three 
conditions:
\begin{equation}\label{Eqn23} 
\begin{array}{cl} 
\textrm{(a)} & 1<p \le 1 + \dfrac{1}{r}, \ \ q \ge 0, \ \ 
0<\tilde{X} < \dfrac{1}{s+1}, \ \ B \ge 0; 
\\[2.0ex] 
\textrm{(b)} & 1+\dfrac{1}{r}<p \le p_c, \ \ q \ge 0, \ \ 
0<\tilde{X} < \dfrac{1}{s+1}, \ \ B \ge 0; 
\\[2.0ex] 
\textrm{(c)} & p > p_c, \ \ q>q_c \ \ 
0<\tilde{X} < \dfrac{1}{s+1}, \ \  B>B_{\rm H}. 
\end{array} 
\end{equation} 

Now directly combining the two conditions in \eqref{Eqn22} and 
three conditions in \eqref{Eqn23} yields the conditions 
(2)-(7) in the theorem as follows:
$$ 
\begin{array}{rll}
\textrm{(i) and (a)} & \Longrightarrow & (2), \\[1.0ex] 
\textrm{(ii) and (a)} & \Longrightarrow & (3), \\[1.0ex] 
\textrm{(i) and (b)} & \Longrightarrow & (4), \\[1.0ex] 
\textrm{(i) and (c)} & \Longrightarrow & (5), \\[1.0ex] 
\textrm{(ii) and (b)} & \Longrightarrow & (6), \\[1.0ex] 
\textrm{(ii) and (c)} & \Longrightarrow & (7). \\[1.0ex] 
\end{array} 
$$
In proving (3), note that $B \ne 0$ since $q_1>0$, 
and $\frac{1}{\tilde{X}_1}>1$ $(0<\tilde{X}_1<1)$. 
For items (3) and (4), 
the condition $\tilde{X} \lessgtr \frac{1}{(s+1)(r+1)}$ comes from 
$$
1 +\dfrac{1}{r} \gtrless \dfrac{1}{\tilde{X}_1} \ \ 
\Longleftrightarrow \ \ 
(s+1)(r+1) \tilde{X} \lessgtr 1. 
$$ 
In proving items (5) and (6), the condition
$\tilde{X} \gtrless \frac{1}{s+1+sr}$ is due to 
$$
p_c \lessgtr \dfrac{1}{\tilde{X}_1} \ \ \Longleftrightarrow \ \ 
(s+1+sr) \tilde{X} \gtrless 1. 
$$ 
Finally, combining (ii) and (c) directly yields (7). 

Therefore, the seven cases in the theorem follow from the 
conditions ${\rm Tr}_1<0$ and $\det_1>0$. 

The two Hopf bifurcation conditions can be obtained from
${\rm Tr_1}=0$ and $\det_1>0$, which are directly by setting
$B=B_{\rm H}$ in items (5) and (7), respectively.
\end{proof}

\section{Two small-amplitude limit cycles when $q=0$}\label{Sec-4}

In this section, we consider the special case $q=0$ and prove that 
the associated Hopf bifurcation is a Bautin bifurcation.
Setting $q=0$ in \eqref{Eqn13} yields 
\begin{equation}\label{Eqn24}
\left\{
\begin{array}{ll}
\dfrac{dx}{dt} = \dfrac{K X^p}{1+B} \left(\dfrac{1}{\tilde{X}}-x-sy\right) 
- r x, \\[2.0ex]     
\dfrac{dy}{dt} = x-y.
\end{array}
\right. 
\end{equation} 
Here, we assume 
$$
p>1, \quad B \ge 0, \quad \textrm{and} \quad 0<s<r<s+1.
$$ 

System \eqref{Eqn24} has a positive equilibrium  
at ${\rm E_1}=(1,1)$. Then, from the first equation 
of \eqref{Eqn24}, we obtain
\begin{equation}\label{Eqn25}
K \big[1-(s+1) \tilde{X} \big] - r (1+B) \tilde{X} = 0,
\end{equation}
and hence
$$  
K = \dfrac{r (1+B) \tilde{X}}{\tilde{X}_1},
$$
where $\tilde{X}_1$ is defined in \eqref{Eqn18}. 
The existence of the positive equilibrium requires
that $0<s<r<s+1$ and $0<\tilde{X}<\frac{1}{s+1}$ as in 
\eqref{Eqn9} and \eqref{Eqn11}. 

The method we will use to prove the existence of limit cycles arising 
from generalized Hopf bifurcations is based on the normal form, or 
equivalently, the focus values or Lyapunov constants. 
We need the following theorem to 
determine the number of small-amplitude limit cycles. This result can 
be found in many textbooks and publications; see, for example, 
\cite{tianyu2015}. 

\begin{lemma}\label{Lem4.1}
Suppose an $n$-dimensional dynamical system $\dot{\boldsymbol{x}} 
= \boldsymbol{f}(\boldsymbol{x},\boldsymbol{\mu})$, 
where $\boldsymbol{x} \in R^n$ and $\boldsymbol{\mu}
\!=\! (\mu_1,\mu_2,\ldots,\,\mu_k)$ is a $k$th-dimensional parameter vector, 
undergoes a Hopf bifurcation from the equilibrium 
$\boldsymbol{x}\!=\!\boldsymbol{0}$ at the critical parameter value 
$\boldsymbol{\mu}\!=\!\boldsymbol{0}$. 
The normal form associated with the Hopf bifurcation 
is given in polar coordinates by 
$$ 
\dot{r} = r \, \big[\, v_0(\boldsymbol{\mu}) + v_1(\boldsymbol{\mu})\, r^2 
+ v_2(\boldsymbol{\mu})\, r^4 + \cdots v_k(\boldsymbol{\mu})\, r^{2k} 
+ \cdots \big], 
$$  
where $r$ represents the amplitude of a limit cycle, and 
$v_j$ is called the $j$th-order focus value. 
Then, the system can have $k$ small-amplitude limit cycles near the 
equilibrium $\boldsymbol{x}=\boldsymbol{0}$ 
for some parameter values $\boldsymbol{\mu}$ sufficiently close to 
$\boldsymbol{0}$, provided that
$$ 
\begin{array}{ll}
v_j(\boldsymbol{0})=0,\quad j=0,\,1,\,\cdots,\,k-1, \ \ v_k(\boldsymbol{0}) 
\ne 0, \\[2.0ex] 
\det\left[\dfrac{\partial(v_0,\,v_1,\,\cdots,\,v_{k-1})} 
{\partial(\mu_1,\,\mu_2,\,\cdots,\,\mu_k)} \right]_{\boldsymbol{\mu=0}} 
\ne 0.  
\end{array} 
$$ 
\end{lemma}

For system \eqref{Eqn24}, we have the following result. 

\begin{proposition}\label{Prop4.2} 
For system \eqref{Eqn24}, two small-amplitude limit cycles can bifurcate
from a Hopf critical point near the positive equilibrium 
${\rm E_1}=(1,\,1)$, giving rise to a Bautin 
$($codimension-two generalized Hopf$\,)$ bifurcation.
Moreover, the outer limit cycle is unstable, whereas the inner limit 
cycle is stable, and both limit cycles enclose the unstable equilibrium 
${\rm E_1}$. 
\end{proposition} 

\begin{proof} 
At the equilibrium ${\rm E_1}=(1,\,1)$, the Jacobian matrix 
of \eqref{Eqn24} is 
$$ 
J(1,\,1) = \left[\begin{array}{cc} 
\dfrac{r [ p \tilde{X}_1 - (1 - s \tilde{X}) ]}{\tilde{X}_1} & 
-\, \dfrac{r s \tilde{X}}{\tilde{X}_1} \\[2.0ex] 
1 & -1 \end{array} \right].  
$$
Consequently, the characteristic polynomial is given by \eqref{Eqn20}, where
$$ 
{\rm Tr}_2 = \dfrac{1}{\tilde{X}_1} 
\big\{p r \tilde{X}_1- \big[ r+1-(s+1+sr) \tilde{X} \big] \big\}
\quad  \textrm{and} \quad 
{\det}_2 = \dfrac{r}{\tilde{X}_1}(1-p \tilde{X}_1) := \omega_c^2.  
$$ 
Since 
$$
r+1- (s+1+sr ) \tilde{X} >0 \quad \textrm{for} \ \ 
0< \tilde{X} < \dfrac{1}{s+1}, 
$$ 
the positive equilibrium ${\rm E_1}$ is LAS if 
$\, {\rm Tr}_2<0 $ and $ {\det}_2>0$, which yield 
$$
p < \min\left\{\dfrac{r+1-(s+1+sr) \tilde{X}}{r \tilde{X}_1}, \, 
\dfrac{1}{\tilde{X}_1} \right\}. 
$$
In order to have a Hopf bifurcation at a critical point, we define 
\begin{equation}\label{Eqn26} 
p_{\rm H} = \dfrac{r+1-(s+1+sr ) \tilde{X}}{r \tilde{X}_1}.
\end{equation}
At this critical point $p_{\rm H}$, it requires $\det_2>0$, yielding 
\begin{equation}\label{Eqn27} 
\dfrac{1}{\tilde{X}_1} > p_{\rm H} 
\ \ \Longrightarrow \ \ 
\dfrac{1}{s+1+sr} < \tilde{X} < \dfrac{1}{s+1},  
\end{equation}
which yields $p_{\rm H}>1$ as expected. 

Moreover, at $p=p_{\rm H}$, the transversal condition is 
$$ 
{\rm Trans}_{\eqref{Eqn24}} = \dfrac{r}{2} > 0,  
$$  
and hence the Hopf bifurcation is transversal.

Next, introducing the affine transformation 
$$ 
\left(\begin{array}{c} x \\ y \end{array} \right) 
= \left(\begin{array}{c} 1 \\ 1 \end{array} \right) 
+ \left[ \begin{array}{cc} 1 & 0 \\[1.0ex] 
\dfrac{\tilde{X}_1}{sr \tilde{X}} & - \dfrac{\omega_c \tilde{X}_1 }
{s r \tilde{X}} \end{array} \right] 
\left(\begin{array}{c} u \\ v \end{array} \right),
$$ 
system \eqref{Eqn24} is transformed into 
\begin{equation}\label{Eqn28} 
\!\!\! 
\left\{ 
\begin{array}{rl}  
\dfrac{d u}{dt} = \!\!\! & \omega_c \, v 
-\, \dfrac{r}{\tilde{X}_1} 
\Big\{ r \tilde{X}_1 (1 +u)  
-\Big[r \tilde{X}_1 - \big(\tilde{X}_1 + r \tilde{X}\big) u 
+ \omega_c \tilde{X}_1 v \Big]
(1+u)^{\frac{r (1- s \tilde{X})+\tilde{X}_1}{r \tilde{X}_1}}, 
\\[3.0ex]  
\dfrac{d v}{dt} = \!\!\! & -\,\omega_c\, u 
- \dfrac{1}{\tilde{X}_1 \omega_c^2} 
\Big\{ r \tilde{X}_1 
+\Big[r (1-\tilde{X})-\tilde{X}_1 \Big] u + \omega_c \tilde{X}_1 v  
\\[2.0ex] 
& \qquad \qquad \qquad \ 
-\Big[r \tilde{X}_1 - (\tilde{X}_1+r \tilde{X}) u 
+ \omega_c \tilde{X}_1 v \Big]
(1+u)^{\frac{r (1-s \tilde{X}) + \tilde{X}_1}{r \tilde{X}_1}} \Big\}.
\end{array} 
\right. 
\end{equation} 

Expanding the right-hand side of \eqref{Eqn28} in 
Taylor series and applying the Maple program for computing 
the normal forms of Hopf and generalized Hopf 
bifurcations~\cite{yu1998,tianyu2013}, we obtain the first focus value 
$$ 
v_1 = \frac{-\,1}{16 r \tilde{X} \tilde{X}_1^3 \big[(s+1+sr) \tilde{X}-1\big]} 
       (r \tilde{X}+ \tilde{X}_1) \big[1+r-(s+1+sr) \tilde{X}) \big]
       \big[ 2 (s+1+sr) \tilde{X}- r - 2 \big].
$$ 
Except for the last factor, all factors in $v_1$ are positive by 
$\frac{1}{s+1+sr}<\tilde{X}<\frac{1}{s+1}$ as given in \eqref{Eqn27}. 
The last factor yields 
$$
2(s+1+sr) \tilde{X} - r -2 \in \left(-r,\, \dfrac{r(s-1)}{s+1} \right). 
$$ 
where $\frac{1}{s+1+sr}<\tilde{X}<\frac{1}{s+1}$ has been used. 
Therefore, when $s>1$, the last factor can have a unique zero, given by 
$$ 
\tilde{X}=\dfrac{r+2}{2(s+1+sr)} \in 
\left\{ \dfrac{1}{s+1+sr},\, \dfrac{1}{s+1} \right\}, \quad (1<s<r<s+1),
$$ 
at which $v_1=0$, and $v_2$ is then simplified as 
$$ 
v_2 \mid_{v_1=0} 
= \dfrac{(s+1)(r+2)^2 (s+1+r) (s+1+sr) \big[s+1+ (2 s-1) r \big]}
{576\, r^4 (s-1)^5} >0 \ \ \ \textrm{for} \ \ \ 1\!<\!s\!<\!r\!<\!s\!+\!1. 
$$ 
This shows that two small-amplitude limit cycles can be obtained 
near the positive equilibrium ${\rm E_1}$ from a Hopf bifurcation 
under proper perturbations on the parameters $p$ and $r$. 
 
This completes the proof of Proposition \ref{Prop4.2}. 
\end{proof}

\section{Generalized Hopf bifurcations generating five small-amplitude 
limit cycles}\label{Sec-5}

In this section, we develop a common focus-value formulation for the 
generalized Hopf analysis and consider successively the existence of three, 
four, and five small-amplitude limit cycles. We then derive necessary 
algebraic and admissibility conditions for codimension six and examine 
whether they can be satisfied.

The successive weak-focus conditions considered in this section are
summarized in Table~\ref{table2}, where $v_0$ denotes the zero-order focus
value associated with the Hopf critical point.

\begin{table}[!h]
\centering
\caption{Focus-value hierarchy for generalized Hopf bifurcations.}
\vspace*{0.15in}
\label{table2}
\begin{tabular}{|c|c|c|}
\hline
Codimension  & Vanishing focus values & First nonzero focus value \\ \hline
3 & $v_0=v_1=v_2=0$ & $v_3\neq 0$ \\
4 & $v_0=v_1=v_2=v_3=0$ & $v_4\neq 0$ \\
5 & $v_0=v_1=v_2=v_3=v_4=0$ & $v_5\neq 0$ \\ \hline
\end{tabular}
\end{table}

The general limit-cycle conclusion will be used through
Corollary~\ref{Coro5.1}, which is a direct specialization of
Lemma~4.1. Thus, the analysis below for each case is organized as follows: 
solving the focus-value equations, checking strict admissibility, 
verifying that the first nonvanishing focus value is nonzero, 
and determining the unfolding rank.

Throughout this section we use the independent parameter set:
\[
\big\{(a,\,m,\,n,\,g,\,D) \mid m>1,\ n>0,\ a>0,\ g>0,\ D>0 \big\}.
\]
For computational purposes, we also use
\begin{equation}\label{Eqn29}
m=1+a+\xi,\qquad n=\xi+\eta,\qquad \xi>0,\qquad \eta>0.
\end{equation}
The parametrization \eqref{Eqn29} does \emph{not} impose
$m-1-n\ge0$, since $m-1-n=a-\eta$ may have either sign.

\subsection{Reparametrization, Hopf conditions, and focus values}\label{Sec-5.1}

The reparametrization used in the symbolic calculation can be written 
directly as
\begin{equation}\label{Eqn30}
p=m,\qquad q=n,\qquad r=\frac{g+1}{a},\qquad s=\frac{1+D}{g},\qquad
\widetilde X=\frac{ag}{(a+1)(g+1)+aD}.
\end{equation}
The two quantities $c$ and $h$, used in the original derivation, 
are completely determined by $(a,m,n)$:
$$
c=m-1-a, \qquad h=a-m+1+n, \qquad c+h=n.
$$
They are auxiliary identities; all focus-value expressions 
below are regarded as functions of $(a,m,n,g,D)$.

At a Hopf critical point,
\begin{equation*}
K=\frac{ng}{a-m+1+n},\qquad
B_{\rm H}=\frac{m-1-a}{a-m+1+n}.
\end{equation*}
Thus, the positivity of $K$ and $B_{\rm H}$ requires
\begin{equation}\label{Eqn31}
m-1-n<a<m-1. 
\end{equation} 
Combining this with the remaining Hopf restrictions gives 
the following admissibility conditions: 
\begin{equation}\label{Eqn32}
\begin{array}{ll}
\mathrm{(A)}&
\displaystyle
\max\left\{m-1-n,\frac{(m-1)(g+1)}{1+D+g}, \frac{g(g+1)}{1+D+g}\right\}
<a< \min\left\{m-1,\frac{g(g+1)}{1+D}\right\},\\[3.0ex]
\mathrm{(B)}&
\displaystyle
\max\left\{m-1-n,\frac{g(g+1)}{1+D+g}\right\}
<a< \min\left\{m-1,\frac{g(g+1)}{1+D}, \frac{(m-1)(g+1)}{1+D+g}\right\}.
\end{array}
\end{equation}
In both cases, $m>1$, $n>0$, $g>0$, and $D>0$ need to be satisfied.

The transversality quantity and the critical frequency become
\begin{equation}\label{Eqn33}
{\rm Trans}=-\,\frac{(g+1)(a-m+1+n)}{2a}<0, \qquad
\omega_c=\frac{n\sqrt D}{a-m+1+n}>0.
\end{equation}

To compute the general focus values of system~\eqref{Eqn13}, we first
multiply its vector field by the positive factor $(1+Bx^q)$, which is 
equivalent to a regular time rescaling in a neighborhood of the positive 
equilibrium.
Using the reparametrization \eqref{Eqn30} and the condition \eqref{Eqn33}, we
then introduce the affine transformation,
$$
\left(\begin{array}{c}x\\y\end{array} \right)
=
\left(\begin{array}{c} 1\\1\end{array} \right)
+
\left[\begin{array}{cc}
1&0\\[0.5ex]
\dfrac{1}{1+D}&-\dfrac{\sqrt D}{1+D}
\end{array} \right] 
\left(\begin{array}{c} x_1\\x_2\end{array} \right),
$$
into \eqref{Eqn13}, yielding the system 
\begin{equation}\label{Eqn34}
\left\{ \  
\begin{aligned}
\dfrac{dx_1}{dt} ={}&\omega_c x_2
+\frac{1}{a(a-m+1+n)}
\Big\{
n\big[(g+1)(1-a x_1)+a\sqrt D\,x_2\big](1+x_1)^m\\
&\hspace{3.5em}
 +(g+1)(a-m+1)(1+x_1)^{n+1}
 -(g+1)(a-m+1+n)(1+x_1)
\Big\},\\[1ex]
\dfrac{dx_2}{dt} ={}&-\omega_c x_1
+\frac{1}{\sqrt D\,a(a-m+1+n)}
\Big\{
n\big[(g+1)(1-a x_1)+a\sqrt D\,x_2\big](1+x_1)^m\\
&\hspace{3.5em}
-\big[aD x_1+(g+1)(1+x_1)+a\sqrt D\,x_2\big]
 \big[n+(m-1-a)(1+x_1)^{n-1}\big] \Big\},
\end{aligned} 
\right. 
\end{equation}
whose linear part is in the standard Jordan canonical form. 

Thus, the Taylor expansion of \eqref{Eqn34} and 
the subsequent normal form computation involve only \((a,m,n,g,D)\).
Applying the Maple programs for computing the normal forms of Hopf and 
generalized Hopf bifurcations (see, e.g.,~\cite{yu1998,tianyu2013}), 
we obtain the first six focus values:  
\begin{equation}\label{Eqn35}
\begin{aligned}
v_1&=-\frac{n}{16(a-m+1+n)a^2D}\,V_1,\\
v_2&= \frac{n}{2304(a-m+1+n)a^4D^3}\,V_2,\\
v_3&=-\frac{n}{4423680(a-m+1+n)a^6D^5}\,V_3,\\
v_4&= \frac{n}{7431782400(a-m+1+n)a^8D^7}\,V_4,\\
v_5&=-\frac{n}{21403533312000(a-m+1+n)a^{10}D^9}\,V_5,\\
v_6&= \frac{n}{2768761069240320000(a-m+1+n)a^{12}D^{11}}\,V_6,
\end{aligned}
\end{equation}
where \(V_k=V_k(a,m,n,g,D)\) denotes the polynomial numerator 
of \(v_k\). All prefactors in \eqref{Eqn35} are nonzero in 
both strict Hopf bifurcation regions. Hence, $v_k=0$ is equivalent to $V_k=0$.

The first focus numerator \(V_1\) is linear in \(D\):
\begin{equation}\label{Eqn36}
V_1= N_D - W \, D, 
\end{equation}
where
\begin{equation}\label{Eqn37}  
\begin{array}{rl} 
N_D=\!\!\!\!& (g+1)\big[a(a+1)+(m-1-a)(a-m+1+n)\big] \\[1.0ex] 
&\times\big[a(a+1)g+(m-1-a)(a-m+1+n)(1+g)\big], \\[1.0ex] 
W=\!\!\!\! & (m-1-a) (g+1) (a-m+1+n)^2 \\[1.0ex]  
&\qquad  -(m-1-a) (a-m+1+n) (gm+m-a) +a^2 (a+1) g. 
\end{array} 
\end{equation}
By \eqref{Eqn31}, we have $N_D>0$. Therefore, whenever $W\ne0$,
\begin{equation}\label{Eqn38}
D=\frac{N_D}{W},\quad D>0 \quad \Longleftrightarrow \quad W>0.
\end{equation}
This exact elimination of $D$ is the starting point for the subsequent
computations. The higher \(V_k\) are lengthy and are omitted for brevity. 

Under \eqref{Eqn29}, Case~$\mathrm{(A)}$ of the Hopf bifurcation 
can be rewritten as 
\begin{equation}\label{Eqn39}
A_1=aD-\xi(g+1),\quad A_2=a(D+g+1)-g(g+1),\quad A_3=g(g+1)-a(D+1),
\end{equation}
with
\begin{equation}\label{Eqn40}
a,\, \xi,\, \eta >0\quad \textrm{and} \ \ 
\Omega_1 := \big\{g>0,\, D>0,\, A_1>0,\, A_2>0,\, A_3>0\big\} >0,
\end{equation}
and $A_2 \!+\! A_3 \!=\! ag \!>\!0$. These margins are particularly 
convenient for the strict-interior constructions.

Substituting $D=\frac{N_D}{W}$ into the Case~$\mathrm{(A)}$ margins
\eqref{Eqn39} gives 
\begin{equation}\label{Eqn41} 
\begin{array}{ll}
A_1 =a\dfrac{N_D}{W}-\xi(g+1) =\dfrac{aN_D-\xi(g+1)W}{W}, \\[2.5ex]
A_2 =a\left(\dfrac{N_D}{W}+g+1\right)-g(g+1)
=\dfrac{aN_D+ (g+1)(a-g)W}{W}, \\[2.5ex]
A_3 =g(g+1)-a\left(\dfrac{N_D}{W}+1\right)
=\dfrac{\big[g(g+1)-a\big] W-aN_D}{W}.
\end{array}
\end{equation} 
Since $N_D$ contains the factor $g+1$, the following are
polynomials in $(a,m,n,g)$:
\begin{equation}\label{Eqn42}
\begin{aligned}
P_1& =\frac{aN_D-\xi(g+1)W}{g+1},\\
P_2& =\frac{aN_D+ (g+1)(a-g)W}{g+1},\\
P_3& =\big[g(g+1)-a\big]W-aN_D.
\end{aligned}
\end{equation}
Consequently,
\begin{equation*}
A_1= \frac{(g+1)P_1}{W}, \qquad A_2= \frac{(g+1)P_2}{W}, \qquad
A_3= \frac{P_3}{W},
\end{equation*}
and hence
\begin{equation}\label{Eqn43}
\Omega_1 >0 \quad \Longleftrightarrow \quad
g>0,\ W>0,\ P_1>0,\ P_2>0,\ P_3>0.
\end{equation}
These polynomial sign conditions will be used directly in
Subsection~\ref{subsec:3LC-poly-admissibility}.

\subsection{Weak-focus conditions and admissible unfolding}\label{Sec-5.2}

The following consequence of Lemma~4.1 will be used repeatedly.

\begin{corollary}\label{Coro5.1}
Let $C_k$, $k\ge2$, be a strict admissible Hopf critical point of
system~\eqref{Eqn13}.  Suppose
\[
v_1(C_k)=\cdots=v_{k-1}(C_k)=0,\qquad v_k(C_k)\ne0,
\]
and let $\mu_2,\ldots,\mu_k$ be the model parameters such that
\begin{equation}\label{Eqn44}
\det\left[
\frac{\partial(v_1,\ldots,v_{k-1})}
{\partial(\mu_2,\ldots,\mu_k)}
\right]_{C_k}\ne0.
\end{equation}
If $\mu_1$ is a Hopf unfolding parameter for which
$\frac{\partial v_0}{\partial\mu_1}\ne0$ at $C_k$, then $C_k$ is a 
nondegenerate generalized Hopf point of codimension $k$. 
Under sufficiently small
admissible perturbations of $(\mu_1,\ldots,\mu_k)$, $k$ distinct
small-amplitude limit cycles bifurcate from ${\rm E}_1$.
\end{corollary}

\begin{proof}
At the Hopf critical point, the full unfolding Jacobian in Lemma~4.1 
factors into the
nonzero Hopf transversality factor $\frac{\partial v_0}{\partial\mu_1}$ 
and the reduced determinant in \eqref{Eqn44}. Hence, the full
Jacobian is nonsingular. Since $C_k$ is a strict interior point of the
admissible Hopf bifurcation region, all defining inequalities remain strict 
under sufficiently small parameter perturbations. Lemma~4.1 therefore 
applies and yields $k$ small-amplitude limit cycles.
\end{proof}

For the cases considered below, the algebraic hierarchy is
\[
\begin{aligned}
C_3:&\quad V_1=V_2=0, &&V_3\ne0,\\
C_4:&\quad V_1=V_2=V_3=0, &&V_4\ne0,\\
C_5:&\quad V_1=V_2=V_3=V_4=0, &&V_5\ne0.
\end{aligned}
\]
A codimension-six generalized Hopf point would additionally require $V_5=0$. 
In every case the focus
equations are supplemented by the strict inequalities in \eqref{Eqn32}.

\subsection{Codimension-three Hopf bifurcation generating 
three small-amplitude limit cycles}\label{Sec-5.3}

In this subsection, we prove the existence of a codimension-three Hopf 
bifurcation, organized as follows. We first reduce the first two focus 
value conditions algebraically and obtain an order-three weak focus 
satisfying Case (A). Then, we characterize the admissible regions in 
Cases (A) and (B) and study the resulting bounds on the exponents $p$ and $q$. 
Finally, we analyze the ordinary and exceptional algebraic branches generated 
by the reduced focus equations. The admissibility conditions are then 
converted into polynomial inequalities, and the exceptional cases are 
ruled out using exact real-algebraic arguments.

We begin with the first genuinely generalized Hopf problem. 
A weak focus of order three is characterized by
\begin{equation*}
V_1=V_2=0,\qquad V_3\ne0,
\end{equation*}
together with a rank-two unfolding condition. The point of the calculation 
is not merely to find a numerical zero of two large expressions.  
The parametrization above makes it possible to reduce the focus equations 
algebraically and, at the same time, to keep the Hopf inequalities 
visible throughout the reduction.

\subsubsection{Algebraic reduction of the first two focus conditions}
\label{Sec-5.3.1}

Since $V_1$ is linear in $D$, Eq.~\eqref{Eqn38} eliminates $D$ 
exactly. Substitution into $V_2$ gives a rational expression 
whose denominator has fixed sign on each strict Hopf bifurcation region.  
Removing this denominator leaves a polynomial equation which is 
quadratic in $g$,
\begin{equation}\label{Eqn45}
Q_2(a,m,n;g)=b_2(a,m,n)g^2+b_1(a,m,n)g+b_0(a,m,n)=0.
\end{equation}
Its positive real branches are separated algebraically by
\begin{equation}\label{Eqn46}
\begin{array}{lll}
\mathrm{(i)} & b_2 b_0 < 0:& \displaystyle
g_+ =\frac{-b_1 + {\rm sign}(b_2) \sqrt{\Delta_g}}{2b_2};\qquad
\Delta_g=b_1^2-4b_2b_0\ge0;\\[2ex]
\mathrm{(ii)} & b_2 b_0 > 0,\ b_2 b_0 b_1<0:& \displaystyle
g_\pm =\frac{-b_1\pm\sqrt{\Delta_g}}{2b_2};\\[2ex]
\mathrm{(iii)} & b_2=0,\ b_1\ne0:& \displaystyle
g=-\frac{b_0}{b_1},\qquad b_1b_0\le0;\\[2ex]
\mathrm{(iv)} & b_2=b_1=0:& b_0=0.
\end{array}
\end{equation}
Using the polynomial reconstruction of the Case~$\mathrm{(A)}$ margins
established above after Eq.~\eqref{Eqn39}, the strict
admissibility conditions remain visible throughout this reduction.

Thus, the simultaneous equations $V_1=V_2=0$ reduce to one explicit 
algebraic branch equation and the reconstruction $D=\frac{N_D}{W}$. 
This reduction is important in the subsequent global analysis: 
admissibility can be tested on the real branches of \eqref{Eqn45}, 
rather than by a blind search in the original five-dimensional 
parameter space.

Under $m=1+a+\xi$, $n=\xi+\eta$, Case~$\mathrm{(A)}$ is described by the 
three margin polynomials \eqref{Eqn39}. Hence, a candidate obtained 
from \eqref{Eqn45} is a strict interior Hopf critical point precisely 
when
\[
a,\ \xi,\  \eta,\  \Omega_1 >0.
\]
This polynomial description is the basis of both the existence proof 
and the exponent optimization below.

\subsubsection{A strict-interior Case~\(\mathrm{(A)}\) point}\label{Sec-5.3.2}

The algebraic reductions below produce highly accurate numerical candidates, 
but a numerical approximation alone does not rigorously prove the existence
of an exact common zero of the focus equations.  We therefore use validated
interval computation to convert such candidates into rigorous statements.
More precisely, we employ the Krawczyk 
method~\cite{krawczyk1969newton,moore2009introduction,rump2010verification}, 
a standard interval technique
which simultaneously certifies existence, local uniqueness, and
nonsingularity of a zero of a nonlinear system.

Let $F:\mathbb{R}^d\to\mathbb{R}^d$. Denote $X\subset\mathbb{R}^d$ as an
interval box with center $x_0$, and let $Y$ be a numerical approximation to
the inverse Jacobian $(F'(x_0))^{-1}$. 
If $[F'(X)]$ is an outward-rounded interval enclosure of
the Jacobian on $X$, define the Krawczyk operator by
\[
K(x_0,X)=x_0-YF(x_0)+\bigl(I-Y[F'(X)]\bigr)(X-x_0).
\]
The strict inclusion
\[
K(x_0,X)\subset X^\circ,
\]
where $X^\circ$ represents the interior of $X$, implies that $F=0$ has a
unique zero in $X$; see, for example,
\cite{moore2009introduction,rump2010verification}.  In the applications
below, $F$ is formed from the required focus equations, while the strict
Hopf inequalities and the first nonzero focus value are verified separately
by outward-rounded interval evaluation on the same certified box. Thus, the
Krawczyk step supplies the rigorous bridge from a numerically located
candidate to an exact weak-focus point.  All interval certificates below use
rational outward-rounded boxes.

\begin{proposition}\label{Prop5.2}
There exist admissible parameters in Hopf Case~\(\mathrm{(A)}\) for which
\[
v_1=v_2=0,\quad v_3\ne0, \quad
\operatorname{rank} \left[ \frac{\partial(v_1,v_2)}{\partial(D,g)}\right]=2.
\]
Consequently, the positive equilibrium is a weak focus of order three 
and three small-amplitude limit cycles bifurcate under sufficiently 
small admissible perturbations.
\end{proposition}

\begin{proof}
Fix the rational slice, 
\[
a=\frac{5626149911}{10^{10}},\qquad
\xi=\frac{6916176059}{10^{10}},\qquad
\eta=\frac{3083823941}{10^{10}}.
\]
Then, 
\[
n=\xi+\eta=1, \qquad m=1+a+\xi=\frac{2254232597}{10^9} 
\]
exactly. With $(D,g)$ as unknowns, denote the center of the Krawczyk box
by $C_3=(D_{C_3},g_{C_3})$. A two-dimensional Krawczyk operator 
applied to the actual normalized equations $v_1=v_2=0$ gives a unique 
zero in the rational box of radius $10^{-34}$ centered at $C_3$, where
\[
D_{C_3}\in\left[\frac{458517237}{683579},\frac{333169691}{496705}\right],
\quad
g_{C_3}\in\left[\frac{6956059}{364615},\frac{12880256}{675143}\right].
\]
Before recording the interval bounds, 
define the five original Case~$\mathrm{(A)}$ 
to quantify the distance from the corresponding boundary conditions:
\begin{equation}\label{Eqn47}
\begin{array}{llll}
T_1 =a-(m-1-n), &\quad  T_2 =a-\dfrac{(m-1)(g+1)}{1+D+g}, & \quad 
T_3 =a-\dfrac{g(g+1)}{1+D+g},\\[2.5ex] 
T_4 =m-1-a, & \quad T_5 =\dfrac{g(g+1)}{1+D}-a.
\end{array}
\end{equation}
Thus, together with positivity of the parameters, strict Case~$\mathrm{(A)}$
admissibility is equivalent to $T_k>0$, $k=1,\ldots,5$.
At this unique zero, interval evaluation gives
\[
\begin{array}{llll}
T_1 \in\left[\dfrac{997}{3233},\dfrac{1100}{3567}\right], & \quad 
T_2 \in\left[\dfrac{416002}{790633},\dfrac{463011}{879976}\right], & \quad  
T_3 \in\left[\dfrac{7246}{888415},\dfrac{3270}{400927}\right],\\[2.5ex] 
T_4 \in\left[\dfrac{2467}{3567},\dfrac{690529}{998426}\right], & \quad  
T_5 \in\left[\dfrac{5335}{702861},\dfrac{581}{76544}\right],
\end{array}
\]
so the certified zero lies strictly inside Case~$\mathrm{(A)}$. 
The equivalent reduced margins satisfy
\[
A_1\in\left[\frac{81059}{223},\frac{2573169}{7079}\right],
\quad A_2\in\left[\frac{30449}{5404},\frac{24437}{4337}\right],
\quad A_3\in\left[\frac{8969}{1759},\frac{29069}{5701}\right].
\]
Moreover, the first nonzero normalized focus value is rigorously enclosed by
\[
v_3\in
\left[\frac{1447}{545884},\frac{1336}{504009}\right],
\]
and the unfolding determinant satisfies
\[
\det \left[ \frac{\partial(v_1,v_2)}{\partial(D,g)}\right] 
\in\left[\frac{-\,1}{1497889},\frac{-\,1}{1497890}\right].
\]
Hence $v_1=v_2=0$, $v_3\ne0$, and the rank is two at a strict 
admissible point.  Lemma~\ref{Lem4.1} with $k=3$ yields three 
small-amplitude limit cycles.
\end{proof}

\subsubsection{Exponent bounds and the geometry of the two Hopf bifurcation 
regions}\label{Sec-5.3.3} 

For Case~$\mathrm{(A)}$, we use the margins \(A_1,A_2,A_3\)
defined in \eqref{Eqn39}; after eliminating \(D\), the originally 
complex admissibility conditions are transformed into polynomial sign 
conditions in \eqref{Eqn40} that are easy to perform symbolic analysis.

For Case~$\mathrm{(B)}$, with the condition given in \eqref{Eqn32}, 
define
\begin{equation}\label{Eqn48}
U_1=\xi(g+1)-aD, \quad U_2=g(g+1)-a(1+D), \quad L=a(1+D+g)-g(g+1).
\end{equation}
Then, 
\begin{equation*}
\Omega_B=
\big\{(a,\,\xi,\,\eta,\,g,\,D) \mid
a, \ \xi, \ \eta, \ g, \ D, \ U_1, \ U_2, \ L>0 \big\},
\end{equation*}
and the remaining original inequality follows from
$$
n(1+D+g)-(m-1)D = (\xi+\eta)(1+D+g)-(a+\xi)D\notag =U_1+\eta(1+D+g)>0.
$$
Hence, the exponent problem is to determine the lower bounds of 
\((m,n)\) on the admissible conditions of the codimension-three 
Hopf bifurcation: 
\begin{equation*}
\inf_{\mathcal C_3\cap\Omega_A}(m,n),
\qquad
\inf_{\mathcal C_3\cap\Omega_B}(m,n),
\qquad
\mathcal C_3:=\{V_1=V_2=0,\ V_3\ne0\}.
\end{equation*}

\subsubsection{The exceptional value \(n=\frac12\)}
\label{subsec:3LC-half-exact}\label{Sec-5.3.4} 

Before considering the Case~$\mathrm{(A)}$ boundary minimum, we dispose of
the exceptional value \(n=\frac12\). With \(n=\xi+\eta\), the
coefficient-degenerate case is precisely condition \(\mathrm{(iv)}\) in
\eqref{Eqn46}, namely
\[
 b_2=b_1=b_0=0.
\]
At \(n=\frac12\), the triple greatest common divisor (gcd) 
of these coefficient equations contains
the factor \(1+2a+2\xi\), which is strictly positive for \(a,\xi>0\).
After removing this known common factor from the primitive polynomial
parts, one obtains
\[
\gcd(b_2,b_1)=1,\qquad
\gcd(b_2,b_0)=1,\qquad
\gcd(b_1,b_0)=\xi(2\xi-1).
\]
Since \(0<\xi<n=\frac12\), none of the remaining factors can vanish.
Exact elimination of \(\xi\) then gives
\[
a^{13}(a+1)^5=0,
\]
which has no solution for \(a>0\). Hence, the exceptional value
\(n=\frac12\) cannot produce an admissible codimension-three point.

The calculations in the next two subsections lead to two distinguished
boundary candidates. The Case~$\mathrm{(A)}$ pair is obtained from the
asymptotic boundary $A_3=0$ together with the reduced weak-focus equations,
whereas the Case~$\mathrm{(B)}$ pair is obtained from the corresponding
double-boundary system. We record the two values here for comparison and
derive them explicitly below.
\begin{equation}\label{Eqn49}
\begin{array}{ll} 
(m_A^*,n_A^*)\approx(1.8919320709,\,0.7959683537), \\[1.0ex] 
(m_B^*,n_B^*)\approx(87.8670322412,\,61.7688995629).
\end{array}
\end{equation}

\subsubsection{Asymptotic boundary mechanism and the Case~\(\mathrm{(A)}\) 
exponent problem}\label{subsec:3LC-global-detail}\label{Sec-5.3.5} 

Along the lower Case~$\mathrm{(A)}$ component,
\begin{equation}\label{Eqn50}
g\to+\infty,
\qquad A_3\to0^+.
\end{equation}
The limiting surface is therefore 
\begin{equation}\label{Eqn51}
A_3=0
\quad\Longleftrightarrow\quad
g(g+1)=a(D+1)
\quad\Longleftrightarrow\quad
D=\frac{g(g+1)}a-1.
\end{equation}
On this surface,
\begin{equation}\label{Eqn52}
A_2=ag,\quad A_1=g(g+1)-a-\xi(g+1)=(g+1)(g-\xi)-a.
\end{equation}
Thus, the strict admissible side of this surface is characterized by 
\begin{equation*}
a>0,\quad \xi>0,\quad \eta>0,\quad g>0, \quad D>0,\quad (g+1)(g-\xi)-a>0.
\end{equation*}

Substituting \eqref{Eqn51} into the reduced 
second-focus equation and imposing stationarity of $ n=\xi+\eta$ 
leads, after elimination, to the following expression: 
\begin{align}\label{Eqn53}
P_{\Psi}(\hat{s})={}&59062500000\hat{s}^{20}-303727809375 \hat{s}^{18}
+775670385375 \hat{s}^{16} -1203893444732 \hat{s}^{14}\notag\\
&+1653881842661 \hat{s}^{12}-1718385644640 \hat{s}^{10}+1291417388777 \hat{s}^8
-759366654992 \hat{s}^6\notag\\
&+324427965219 \hat{s}^4-91173112245 \hat{s}^2+14601770640.
\end{align}
Here, the elimination variable $\hat{s}$ is the stationary value of
$n=\xi+\eta$ on the boundary family; thus, for every admissible root
of \eqref{Eqn53}, one has $\hat{s}=n$ at the corresponding stationary
boundary point.  Its real roots are
\begin{equation*}
\hat{s}=\pm0.7959683537\cdots, \qquad \hat{s}=\pm0.8277388406\cdots.
\end{equation*}
Since $n>0$, only the two positive roots need be considered.  
Back-substitution
into the boundary equations and the strict Case~$\mathrm{(A)}$ sign
conditions identifies $\hat{s}=0.7959683537\ldots$ with the lower admissible
branch relevant to the minimum-exponent problem; the root
$\hat{s}=0.8277388406\ldots$ belongs to the other stationary branch and does
not give the lower boundary value.  Consequently,
\begin{equation}\label{Eqn54}
n_A^*=\hat{s}=0.7959683537\cdots.
\end{equation}
Back-substitution of this root into the remaining boundary equations
then gives
\begin{equation*}
m_A^*=1.8919320709\cdots.
\end{equation*}
Since \eqref{Eqn50} holds,
\begin{equation*}
n\downarrow n_A^*,\qquad g\uparrow\infty, \qquad A_3\downarrow0,
\end{equation*}
so the boundary value is an infimum rather than a finite strict point, 
as shown in Figure~\ref{Fig1}.

\begin{figure}[!h]
\centering
\includegraphics[width=0.72\textwidth]{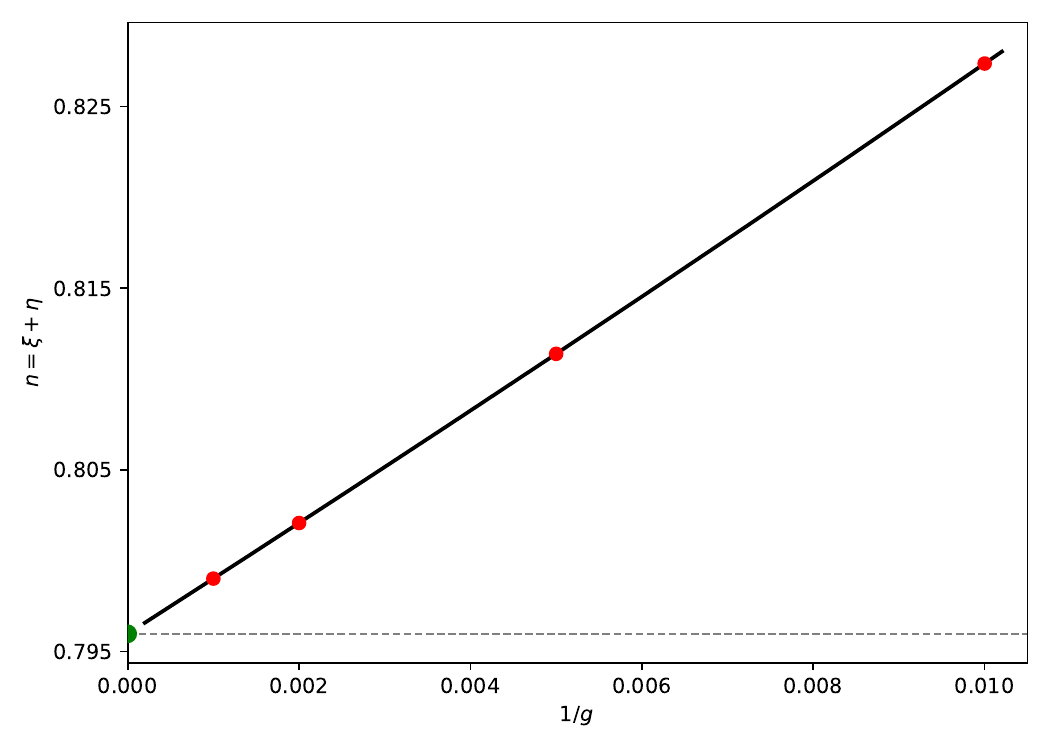}
\caption{Case~$\mathrm{(A)}$ codimension-three Hopf bifurcation with 
finite continuation toward the asymptotic 
boundary candidate. The ordinate is $n=\xi+\eta$ and the abscissa is 
$1/g$, so the limiting regime $g\to\infty$ is represented by the finite 
endpoint $1/g=0$. The solid curve joins the retained finite admissible 
continuation values. The dashed horizontal line is the isolated boundary 
stationary value $n_A^*$, and the diamond at $1/g=0$ marks this limiting 
boundary value. Along the computed branch, $n$ approaches $n_A^*$ from 
above while $A_3\to0^+$. The figure illustrates the asymptotic boundary 
mechanism behind the Case~$\mathrm{(A)}$ optimization; it does not 
by itself establish that $n_A^*$ is the unconditional global infimum 
over all finite strict Case~$\mathrm{(A)}$ branches.}
\label{Fig1}
\end{figure}

\begin{proposition}\label{Prop5.3}
Every finite strict Case~\(\mathrm{(A)}\) codimension-three point 
below the asymptotic threshold \eqref{Eqn54} lies on one of 
three algebraic branches. Hence, exact exclusion of these branches 
below \(n_A^*\) excludes disconnected finite components there.
\end{proposition}

\begin{proof}
Using directly the original exponent variable $n$, write
\[
n=\xi+\eta,\qquad \eta=n-\xi,\qquad 0<\xi<n.
\]
The reduced second-focus equation is the quadratic
\[
F_n(a,\xi,n,g)
=b_2(a,\xi,n)g^2+b_1(a,\xi,n)g+b_0(a,\xi,n)=0, 
\]
which has already been classified in \eqref{Eqn46}. Accordingly,
\begin{equation}\label{Eqn55}
\mathcal V(F_n)=\mathcal B_2\cup\mathcal B_1\cup\mathcal B_0,
\end{equation}
where
\begin{align}
\mathcal B_2&=\{F_n=0,\ b_2\ne0\},\label{Eqn56}\\
\mathcal B_1&=\{b_2=0,\ b_1\ne0,\ g=-b_0/b_1\},\notag\\
\mathcal B_0&=\{b_2=b_1=b_0=0\}.\label{Eqn57}
\end{align}
Thus, every finite point belongs to one of these three branches.  The
decomposition \eqref{Eqn55}-\eqref{Eqn57} is the algebraic
counterpart of the explicit root classification \eqref{Eqn46}: on
$\mathcal B_2$ one has $b_2\ne0$ and the quadratic roots are selected by
conditions \(\mathrm{(i)}\)--\(\mathrm{(ii)}\) in \eqref{Eqn46};
$\mathcal B_1$ corresponds to condition \(\mathrm{(iii)}\), and
$\mathcal B_0$ to condition \(\mathrm{(iv)}\). Hence, exact
semialgebraic exclusion of $\mathcal B_2,\mathcal B_1,\mathcal B_0$ below
$n_A^*$ excludes any disconnected finite component in that range.
\end{proof}

\subsubsection{Case~\(\mathrm{(B)}\): double-boundary geometry 
and an exact interior certificate}\label{subsec:3LC-B-detail}\label{Sec-5.3.6} 

For Case~$\mathrm{(B)}$, the strict inequalities are
\begin{equation*}
U_1>0,\qquad U_2>0,\qquad L>0,
\end{equation*}
where the expressions of $U_1$, $U_2$ and $L$ are given 
in \eqref{Eqn48}. On the first-focus locus one further obtains
\begin{equation}\label{Eqn58}
aD>(a+1)(g+1).
\end{equation}
Combining \eqref{Eqn58} with $U_1>0$ yields
\begin{equation*}
\xi(g+1)>aD>(a+1)(g+1) \quad \Longrightarrow \quad \xi>a+1,
\end{equation*}
and therefore
\begin{equation*}
m=1+a+\xi>2(a+1), \qquad n=\xi+\eta>a+1.
\end{equation*}
Also, \eqref{Eqn58} and $U_2>0$ imply
\begin{equation}\label{Eqn59}
g(g+1)>a(1+D) >a+ (a+1)(g+1),
\end{equation}
which leads to 
\begin{equation}\label{Eqn60}
g^2-ag-2a-1>0 \quad \Longrightarrow \quad 
 g>g_+(a) := \frac{a+\sqrt{a^2+8a+4}}2>a+1.
\end{equation}
These inequalities force the Case~$\mathrm{(B)}$ codimension-three 
locus toward large exponents. The active double-boundary system 
may be written schematically as 
\begin{equation*}
 \widehat V_2=0,
 \qquad B_1=0,
 \qquad B_2=0,
\end{equation*}
where $B_1,B_2$ are the two active boundary polynomials obtained from 
\eqref{Eqn32}. Eliminating the auxiliary variables gives the relevant 
real solution, 
\begin{equation*}
 n_B^*=61.7688995629\cdots,
 \qquad
 m_B^*=87.8670322412\cdots.
\end{equation*}
A strict-interior certificate then verifies
\begin{equation*}
V_1=V_2=0,
\qquad V_3\ne0,
\qquad U_1,\,U_2,\,L>0.
\end{equation*}
The above result is illustrated in Figure~\ref{Fig2}.

\begin{figure}[!h]
\centering
\includegraphics[width=0.72\textwidth]{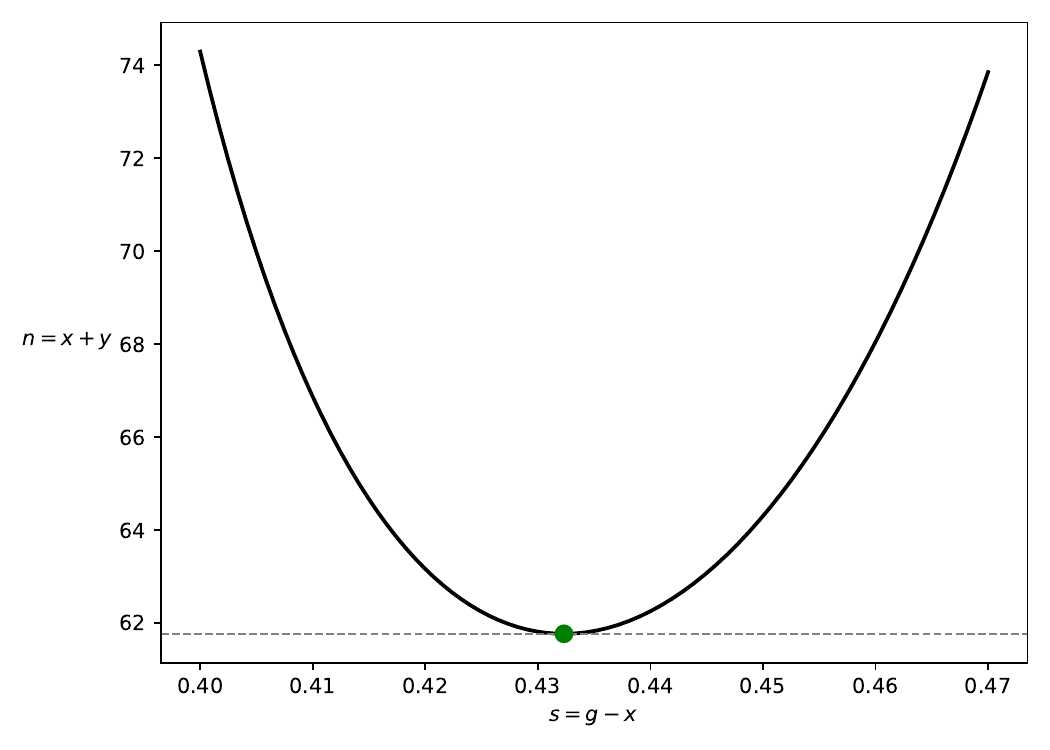}
\caption{Case~$\mathrm{(B)}$ codimension-three Hopf bifurcation with 
stationary candidate on the double-boundary 
algebraic branch. The solid curve plots the relevant real double-boundary 
component through the numerically identified stationary point, 
with ordinate $n=\xi+\eta$. The diamond marks the stationary candidate 
$(s_B^*,n_B^*)$, and the dashed line indicates the corresponding candidate 
boundary value of $n$. The curve displays the local minimum of $n$ on this 
double-boundary component. Since the double boundary lies outside the 
strict Case~$\mathrm{(B)}$ region and the remaining strict and 
single-boundary branches have not all been globally excluded, 
the marked value is a boundary-optimization candidate rather than 
a claimed global infimum for the finite strict Case~$\mathrm{(B)}$ problem.}
\label{Fig2}
\end{figure}

\subsubsection{Ordinary and exceptional branches in the global 
Case~\(\mathrm{(A)}\) reduction}\label{subsec:3LC-branches}\label{Sec-5.3.7} 

On $\mathcal B_2$, the real roots are given in \eqref{Eqn46}. 
Note that here only those real roots that satisfy 
$g_{\pm}>0$ are admissible. The condition $g_{\pm}>0$ is imposed directly 
on each real root and is not inferred from the signs of the coefficients $b_2$,
$b_1$ and $b_0$; in particular, neither $b_1<0$ nor positivity of both
roots is assumed.

Modulo $F_n$, every polynomial in $g$ reduces to degree at most one:
\begin{equation*}
g^2\equiv-\frac{b_1}{b_2}g-\frac{b_0}{b_2}
\pmod{F_n},
\end{equation*}
and hence each cleared admissibility numerator has the form
\begin{equation*}
M(a,\xi,n,g)
\equiv M_1(a,\xi,n)g+M_0(a,\xi,n)
\pmod{F_n}.
\end{equation*}
Thus, its zero set on the quadratic branch is projected by
\begin{equation*}
\mathcal R_M(a,\xi,n)
:=\operatorname{Res}_g(F_n,M_1g+M_0)=0.
\end{equation*}
When $M_1\ne0$,
\begin{equation*}
\mathcal R_M
=b_2M_0^2-b_1M_0M_1+b_0M_1^2,
\end{equation*}
up to a nonzero factor/sign convention for the resultant.  
This removes the square root $g=g_\pm$. 

On $\mathcal B_1$,
\begin{equation*}
b_2=0,
\qquad b_1\ne0,
\qquad g=-\frac{b_0}{b_1}>0,
\end{equation*}
and all margins reduce to rational functions of $(a,\xi,n)$.

On $\mathcal B_0$,
\begin{equation*}
 b_2=b_1=b_0=0,
 \qquad
 a>0,
 \qquad
 0<\xi<n,
\end{equation*}
so $F_n$ imposes no condition on $g$. Therefore, no division 
by $b_2$ or $b_1$ is permitted on this component. The coefficient ideal
\begin{equation*}
\mathcal I_0=\langle b_2,b_1,b_0\rangle
\end{equation*}
is decomposed and its real components are checked separately. 
The special value $n=\frac{1}{2}$ has already been excluded exactly in
Subsection~\ref{subsec:3LC-half-exact}.

\subsubsection{Polynomial admissibility after 
eliminating $D$}\label{subsec:3LC-poly-admissibility}\label{Sec-5.3.8}

Using the polynomials $N_D,W,P_1,P_2,P_3$ introduced in
Subsection~5.3.1, the strict Case~$\mathrm{(A)}$ admissibility conditions
are exactly the polynomial sign conditions in
\eqref{Eqn43}. Thus, no further rational inequalities need
to be introduced here.
Consequently, the strict codimension-three Case~$\mathrm{(A)}$ problem is
\begin{equation}\label{Eqn61}
F_n=0, \quad a,\, \xi,\, \eta,\, g>0,
\quad W>0, \quad P_1>0, \quad P_2>0, \quad P_3>0.
\end{equation}
Equivalently, with $n=\xi+\eta$,
\begin{equation*}
 F_n(a,\xi,n,g)=0,
 \quad a>0,
 \quad 0<\xi<n,
 \quad g>0,
 \quad W>0,
 \quad P_1>0,
 \quad P_2>0,
 \quad P_3>0.
\end{equation*}
The branch decomposition \eqref{Eqn56}-\eqref{Eqn57} together 
with \eqref{Eqn61} is the exact polynomial 
formulation used for the global Case~$\mathrm{(A)}$ analysis.

\subsubsection{Completion of the exceptional-branch exclusion}
\label{subsec:3LC-exceptional-exact}\label{Sec-5.3.9} 

It remains to exclude the rest of the coefficient-degenerate branch
\(b_2=b_1=b_0=0\). The special value \(n=\frac12\) was treated exactly in
Subsection~\ref{subsec:3LC-half-exact}. The following lemma covers the
remaining exceptional component in the range relevant to the
Case~$\mathrm{(A)}$ threshold.

\begin{lemma}\label{Lem5.4}
The exceptional branch $ b_2=b_1=b_0=0 $ has no solution in
$ a>0$, $0<\xi<n<\frac{4}{5}$.
\end{lemma}

\begin{proof}
The subbranch \(n=\frac12\) has no admissible solution by
Subsection~\ref{subsec:3LC-half-exact}. For the remaining exceptional
component, we substitute \(\eta=n-\xi\) into
\(b_2=b_1=b_0=0\), remove the factor \(2n-1\) considered above, and
compute a lexicographic Gröbner basis for the resulting polynomial ideal.
Eliminating successively \(\xi\) and \(n\) yields the following three
relations:
\[
P_c(a)=334a^3+134a^2+6a-1=0,
\]
together with
\begin{equation}\label{Eqn62}
367296a^2+9583n^2+90168a-8863=0, 
\end{equation}
and
\begin{equation}\label{Eqn63}
-1336a^2-470a-259n+518\xi+183=0.
\end{equation}
The cubic \(P_c\) has exactly one positive real root.  Put
\[
T(a)=183-1336a^2-470a.
\]
Then \eqref{Eqn63} gives
\begin{equation}\label{Eqn64}
518\xi=259n-T(a).
\end{equation}
Using \eqref{Eqn62}, reduce the numerator of
\(259^2n^2-T(a)^2\) modulo \(P_c\). Up to a nonzero rational 
factor the remainder is
\[
 S_{\rm num}(a)=
 -446224a^4-313960a^3-575749a^2-114789a+7138.
\]
Exact Sturm calculations on a rational isolating interval for the 
positive root of \(P_c\) yield
\[
T(a)>0,\qquad S_{\rm num}(a)<0.
\]
Consequently, it gives 
\[
259^2n^2<T(a)^2.
\]
Since \(n>0\) and \(T(a)>0\), we have 
\[
 259n<T(a),
\]
and \eqref{Eqn64} implies \(\xi<0\), contradicting admissibility.  
Hence, the exceptional branch contains no point with
\(a>0\) and \(0<\xi<n<\frac{4}{5}\). The rational cutoff \(\frac{4}{5}\) 
is chosen because the Case~\(\mathrm{(A)}\) asymptotic threshold satisfies
\(n_A^*=0.7959683536\cdots< \frac{4}{5}\). Hence, this exclusion covers 
the entire range below the threshold relevant 
to Proposition~\ref{Prop5.3}.
\end{proof}

\subsection{Codimension-four Hopf bifurcation generating 
four small-amplitude limit cycles}\label{Sec-5.4}

The codimension-four problem requires
\[
 V_1=V_2=V_3=0,\qquad V_4\ne0,
\]
together with three independent unfolding directions.  We first locate the
relevant algebraic components of the reduced focus system, then rigorously
certify a strict-interior Case~$\mathrm{(A)}$ solution, verify the
nondegeneracy of the first three focus values, and then apply
Corollary~\ref{Coro5.1}.

\subsubsection{Algebraic reduction and localization of 
the codimension-four candidates}\label{Sec-5.4.1} 

After eliminating $D$ by \eqref{Eqn38}, the equations
$V_2=V_3=0$ define a substantially thinner algebraic set than in the
three limit cycles problem. With $ m=1+a+\xi,\, n=\xi+\eta$, 
we use the Case~$\mathrm{(A)}$ margin, $ A_3=g(g+1)-a(D+1)$ 
to localize the low-exponent component.  On the boundary $A_3=0$,
\[
D=\frac{g(g+1)}a-1,
\]
so the reduced focus system loses one variable.  A stationary point of
$n=\xi+\eta$ on this boundary occurs near
\[
a=0.4666024292\cdots,\ \ \xi=0.5070197806\cdots,\ \ 
\eta=0.3706180362\cdots,\ \ g=54.0525566300\cdots,
\]
with
\[
m=1.9736222100\cdots,\qquad n=0.8776378168\cdots.
\]
This boundary calculation is used only to locate the low-exponent component;
the boundary point itself is not admissible for the 
unfolding of four limit cycles.
We therefore move into the strict interior by fixing exactly
\begin{equation}\label{Eqn65}
g=\frac{13513139162293}{250000000000},\qquad
A_3=\frac{314}{25},
\end{equation}
and hence
\[
D=\frac{g(g+1)-A_3}{a}-1.
\]
The resulting three equations $V_1=V_2=V_3=0$ are then solved in
$(a,\xi,\eta)$ and certified below.

\subsubsection{Case~\(\mathrm{(B)}\) and structural exponent bounds}
\label{Sec-5.4.2} 

For Case~$\mathrm{(B)}$, the same reduced focus system is used. The strict
Case~$\mathrm{(B)}$ inequalities imply, on the first-focus locus,
\[
aD>(a+1)(g+1),
\]
and therefore
\[
 \xi>a+1,\qquad m=1+a+\xi>2a+2,\qquad n=\xi+\eta>a+1.
\]
Moreover, under the conditions \eqref{Eqn59} and \eqref{Eqn60}, 
we obtain that every Case~$\mathrm{(B)}$ codimension-four 
candidate satisfies the structural lower bounds
\begin{equation}\label{Eqn66}
{m>2a+2,\qquad n>a+1,\qquad g>a+1.}
\end{equation}
A multistart search of the reduced equations locates a large-exponent
Case~$\mathrm{(B)}$ branch near
\[
(a,\xi,\eta)=(27.8179055000\cdots,\,61.2225000000\cdots,\,
 1.6415100000\cdots),
\]
corresponding to
\[
 m=90.0400000000\cdots,\qquad n=62.8640000000\cdots.
\]
This branch is recorded to describe the second Hopf bifurcation region; 
the rigorous four-cycle existence result below is obtained from the 
low-exponent strict-interior Case~$\mathrm{(A)}$ point.

\subsubsection{Rigorous certification of the Case~\(\mathrm{(A)}\) weak focus}
\label{subsec:4LC-certification}\label{Sec-5.4.3} 

\begin{proposition}\label{Prop5.5}
The generalized SIRS system \eqref{Eqn13} possesses a strict-interior weak
focus $C_4$ of order four in Hopf Case~$\mathrm{(A)}$.  More precisely,
\[
 v_1(C_4)=v_2(C_4)=v_3(C_4)=0,\qquad v_4(C_4)>0.
\]
\end{proposition}

\begin{proof}
Fix the exact values in \eqref{Eqn65}. We use the Krawczyk criterion
introduced before Proposition~\ref{Prop5.2}, now applied to the
three equations $V_1=V_2=V_3=0$ in $(a,\xi,\eta)$. For the rational box
\[
 X=C_4+[-10^{-30},10^{-30}]^3,
\]
the interval computation gives
\[
 K(X)\subset X^\circ,
\]
where $X^\circ$ denotes the interior of $X$.  Hence, $X$ contains a unique
common zero.  In convenient rational interval form,
\begin{equation}\label{Eqn67}
\begin{aligned}
a_{C_4} \in\left[\frac{322954}{691989},\frac{196109}{420200}\right],\quad
\xi_{C_4} \in\left[\frac{374637}{738350},\frac{240647}{474277}\right],\quad
\eta_{C_4} \in\left[\frac{354723}{958022},\frac{112502}{303841}\right].
\end{aligned}
\end{equation}
The reconstructed parameters satisfy
\[
\begin{aligned}
D \in\left[\frac{2701142395}{425501},\frac{5079481575}{800152}\right],\quad
m \in\left[\frac{1330882}{674171},\frac{721771}{365620}\right],\quad
n \in\left[\frac{555088}{632461},\frac{461221}{525510}\right].
\end{aligned}
\]
On the same certified box, the five Case~$\mathrm{(A)}$ margins satisfy
\begin{equation}\label{Eqn68}
\begin{array}{lll}
T_1\in\left[\dfrac{354723}{958022},\dfrac{112502}{303841}\right],&
T_2\in\left[\dfrac{374815}{817786},\dfrac{178691}{389875}\right],&
T_3\in\left[\dfrac{1913}{967061},\dfrac{1235}{624318}\right],\\[2.5ex]
T_4\in\left[\dfrac{374637}{738350},\dfrac{240647}{474277}\right],&
T_5\in\left[\dfrac{477}{241126},\dfrac{1813}{916481}\right].
\end{array}
\end{equation}
All lower endpoints are positive, so $C_4$ lies strictly inside
Case~$\mathrm{(A)}$.  Finally, the first nonzero normalized focus value is
rigorously enclosed by
\begin{equation}\label{Eqn69}
v_4(C_4)\in
\left[\frac{194}{943439},\frac{137}{666243}\right]
\subset \big(0,\,2.1 \! \times \! 10^{-4} \big).
\end{equation}
Therefore $v_1=v_2=v_3=0$ and $v_4>0$ at the unique certified point $C_4$.
\end{proof}

\subsubsection{Nondegeneracy and admissible local unfolding}\label{Sec-5.4.4}

\begin{lemma}\label{Lem5.6}
At $C_4$, with $m$ and $n$ fixed, the map
$ (a,g,D)\longmapsto(v_1,v_2,v_3) $
has nonsingular derivative.  Consequently, $(v_1,v_2,v_3)$ form local
coordinates for an admissible three-parameter unfolding of the order-four
weak focus.
\end{lemma}

\begin{proof}
Write $v_j=c_jV_j$, $j=1,2,3$, where the normalization factors $c_j$ are
nonzero on the certified box.  At $C_4$, since $V_1=V_2=V_3=0$, 
it follows directly that
\[
 \frac{\partial(v_1,v_2,v_3)}{\partial(a,g,D)}(C_4)
 =\operatorname{diag}\bigl(c_1(C_4),c_2(C_4),c_3(C_4)\bigr)
 \frac{\partial(V_1,V_2,V_3)}{\partial(a,g,D)}(C_4).
\]
The normalized determinant satisfies the rigorous interval enclosure
\begin{equation}\label{Eqn70}
J_4(C_4) =
\det\left[\frac{\partial(v_1,v_2,v_3)}{\partial(a,g,D)}\right]_{C_4}
\in\left[\frac{2}{5}\times10^{-9},\frac{1}{2}\times10^{-9}\right]
\subset(0,10^{-9}).
\end{equation}
Thus, $J_4(C_4)\ne0$.  The inverse-function theorem implies that
$(v_1,v_2,v_3)$ may be prescribed independently and sufficiently small by
small variations of $(a,g,D)$ with $m,n$ fixed.

It remains to verify that these variations can be made without leaving the
Hopf bifurcation region. Let
\[
 \rho_4 =\min_{1\le j\le5}\inf T_j(C_4)>0,
\]
where positivity follows from \eqref{Eqn68}. By continuity, there is
a neighborhood $U$ of $C_4$ such that
\[
 T_j(z)>\frac{\rho_4}{2}>0,\qquad j=1,\ldots,5,
\]
for every $z\in U$.  Shrinking the inverse-function neighborhood if
necessary, the entire local unfolding is therefore contained in strict
Case~$\mathrm{(A)}$.
\end{proof}

By Proposition~\ref{Prop5.5} and
Lemma~\ref{Lem5.6}, all hypotheses of
Corollary~\ref{Coro5.1} with $k=4$ are satisfied. Hence, $C_4$ is a
nondegenerate codimension-four generalized Hopf point and four distinct
small-amplitude limit cycles bifurcate from ${\rm E}_1$ under sufficiently
small admissible perturbations.

\subsection{Codimension-five Hopf bifurcation generating five
small-amplitude limit cycles}\label{Sec-5.5}

This subsection contains the main bifurcation result of the paper.  The
equation $V_1=0$ is first solved exactly for $D$, and resultant elimination
is used to locate a codimension-five component of the first four focus
equations.  A nearby exact rational slice is then chosen strictly inside
Hopf Case~$\mathrm{(A)}$, and a Krawczyk inclusion converts the numerical
candidate into a rigorous isolated zero.  Finally, interval evaluation
certifies strict admissibility, $v_5\ne0$, and the full unfolding rank.
Corollary~\ref{Coro5.1} then gives the five-cycle conclusion. Thus,
the numerical search is used only for localization; the existence theorem
itself rests on exact algebraic reduction and validated interval computation.

\subsubsection{Algebraic reduction and localization of the codimension-five
component}\label{Sec-5.5.1}

With $ m=1+a+\xi$ and $ n=\xi+\eta$, the first focus equation is linear 
in $D$ after the nonzero factors have been removed.  
Substitution of $D=\frac{N_D}{W}$ into the remaining focus equations gives
rational functions of $(a,\xi,\eta,g)$.  Let $\widehat v_j$ denote the
numerator obtained from
\[
v_j\big|_{D=\frac{N_D}{W}},\qquad j=2,3,4,5,
\]
after clearing the denominators that are nonzero in the admissible region.
The first two projected equations are
\[
F_{23}=\operatorname{Res}_g(\widehat v_2,\widehat v_3),\qquad
F_{24}=\operatorname{Res}_g(\widehat v_2,\widehat v_4).
\]
After removing factors belonging to degenerate or center-type branches, the
reduced resultants $F_{23a}$ and $F_{24a}$ satisfy
\begin{equation}\label{Eqn71}
F_{23a}(a,\xi,\eta)=F_{24a}(a,\xi,\eta)=0.
\end{equation}
Equation \eqref{Eqn71} is used only to locate candidate components.
Since resultant equations may contain projection artifacts, every candidate
is used to recover a common value of $g$ and checked in the original focus
equations.  This computation locates a strict Case~$\mathrm{(A)}$ component
near
\[
a\approx0.0272866424,\quad
\xi\approx0.0000223054,\quad
\eta\approx3.2848526393.
\]
A further univariate elimination produces prohibitively large intermediate
expressions and is not needed for the certification below.

\subsubsection{Exact rational slice and Krawczyk certification}
\label{Sec-5.5.2}

For a rigorous certificate we choose the exact rational value
\begin{equation}\label{Eqn72}
a=\frac{321}{11764}.
\end{equation}
This value is a rational point close to the numerically located component but
strictly inside the admissible region.  
Therefore, it is reasonable to use exact outward rounding interval arithmetic, 
which remains a positive distance from the admissibility boundary 
for the subsequent unfolding.

At \eqref{Eqn72}, solve \eqref{Eqn38} 
exactly for $D=D(\xi,\eta,g)$ and define the matrix (square) system
\[
F(\xi,\eta,g)=
\bigl(\widetilde v_2,\widetilde v_3,\widetilde v_4\bigr),\qquad
\widetilde v_j(\xi,\eta,g)
=v_j(D(\xi,\eta,g),\xi,\eta,g).
\]
A three-dimensional Krawczyk computation is carried out on a rational box
$\mathcal B_*$ of radius $10^{-38}$ about the numerical candidate. If $x_0$
is the center of $\mathcal B_*$, $[JF(\mathcal B_*)]$ is an interval enclosure
of the Jacobian, and $Y\approx J(F(x_0))^{-1}$ is a point preconditioner, then
\[
K(x_0,\mathcal B_*)
=x_0-YF(x_0)+\bigl(I-Y[JF(\mathcal B_*)]\bigr)(\mathcal B_*-x_0).
\]
The validated computation gives
\begin{equation}\label{Eqn73}
K(x_0,\mathcal B_*)\subset\mathcal B_*^{\circ},
\end{equation}
where $\mathcal B_*^{\circ}$ denotes the interior of $\mathcal B_*$.  Hence
$F=0$ has a unique zero $(\xi_*,\eta_*,g_*)\in\mathcal B_*$. By the exact
reconstruction \eqref{Eqn38}, 
the corresponding $D_*$ also satisfies $v_1=0$.

For convenient independent verification, the certified zero and reconstructed
parameters are contained in the following simpler rational outer enclosures:
\begin{equation}\label{Eqn74}
\begin{array}{lll}
\displaystyle
\xi_*\in\left[\frac{1652}{74062435},\frac{1538}{68951589}\right],&
\displaystyle
\eta_*\in\left[\frac{3164903}{963484},\frac{3098457}{943256}\right],&
\displaystyle
g_*\in\left[\frac{300135}{429064},\frac{580385}{829701}\right],\\[2.5ex]
\displaystyle
D_*\in\left[\frac{19732763}{467401},\frac{41246917}{976997}\right],&
\displaystyle
m_*\in\left[\frac{468307}{455858},\frac{112177}{109195}\right],&
\displaystyle
n_*\in\left[\frac{146489}{44595},\frac{2905738}{884581}\right].
\end{array}
\end{equation}

\subsubsection{Strict admissibility, fifth focus value, and unfolding rank}
\label{Sec-5.5.3}

For the Case~$\mathrm{(A)}$ margins $T_1,\ldots,T_5$ defined in
\eqref{Eqn47}, interval evaluation on the certified box gives
\begin{equation}\label{Eqn75}
\begin{array}{lll}
\displaystyle
T_1\in\left[\frac{3164903}{963484},\frac{3098457}{943256}\right],&
\displaystyle
T_2\in\left[\frac{13629}{519599},\frac{14657}{558791}\right],&
\displaystyle
T_3\in\left[\frac{217}{999373},\frac{116}{534227}\right],\\[2.5ex]
\displaystyle
T_4\in\left[\frac{1}{44832},\frac{22}{986303}\right],&
\displaystyle
T_5\in\left[\frac{49}{221719},\frac{122}{552035}\right].
\end{array}
\end{equation}
Every lower endpoint in \eqref{Eqn75} is positive. Thus, the
certified zero lies strictly inside Hopf Case~$\mathrm{(A)}$.

At the same zero,
\begin{equation}\label{Eqn76}
v_1=v_2=v_3=v_4=0,
\quad
v_5\in
\left[-\frac{1}{1896000000},-\frac{1}{1898000000}\right]
\subset \left(-0.5 \! \times \! 10^{-9},\, 0 \right).
\end{equation}
Consequently, the positive equilibrium is a weak focus of order five.

It remains to verify that the first four focus values provide four independent
unfolding directions. Since $v_1=0$ has been solved regularly for $D$, the
implicit differentiation of
$v_1(D(\xi,\eta,g),\xi,\eta,g)=0$ gives the Schur-complement factorization
\begin{equation}\label{Eqn77}
\det\left[\frac{\partial(v_1,v_2,v_3,v_4)}
{\partial(D,\xi,\eta,g)}\right]
= \frac{\partial v_1}{\partial D}
\det\left[\frac{\partial(\widetilde v_2,\widetilde v_3,\widetilde v_4)}
{\partial(\xi,\eta,g)}\right].
\end{equation}
Here, $\partial_Dv_1$ is separated from zero on the certified box, and the
reduced determinant satisfies the rigorous interval enclosure
\begin{equation}\label{Eqn78}
\det\left[\frac{\partial(\widetilde v_2,\widetilde v_3,\widetilde v_4)}
{\partial(\xi,\eta,g)}\right] 
\in \left[-\frac{34023}{19},-\frac{39395}{22}\right] \subset(-1791,\,0).
\end{equation}
Therefore, \eqref{Eqn77} is nonzero and
\begin{equation}\label{Eqn79}
\operatorname{rank}
\left[\frac{\partial(v_1,v_2,v_3,v_4)}
{\partial(D,\xi,\eta,g)}\right] =4.
\end{equation}
By the inverse-function theorem, $(v_1,v_2,v_3,v_4)$ may be used as four
independent local unfolding parameters near the certified point, with $a$
fixed by \eqref{Eqn72}. Because all inequalities in
\eqref{Eqn75} are strict, sufficiently small unfolding
perturbations remain inside the admissible Case~$\mathrm{(A)}$ region.

\subsubsection{Main result of codimension-five Hopf bifurcation}
\label{Sec-5.5.4}

\begin{proposition}\label{Prop5.7}
For the exact value $a=\frac{321}{11764}$, there exists a strict-interior point
$C_5$ in Hopf Case~$\mathrm{(A)}$ such that
\[
v_1(C_5)=v_2(C_5)=v_3(C_5)=v_4(C_5)=0,
\qquad v_5(C_5)\ne0,
\]
and
\[
\det\left[\frac{\partial(v_1,v_2,v_3,v_4)}
{\partial(D,\xi,\eta,g)}(C_5)\right] \ne0.
\]
Consequently ${\rm E}_1$ is a nondegenerate weak focus of order five and,
under sufficiently small admissible perturbations, five distinct
small-amplitude limit cycles bifurcate from ${\rm E}_1$.
\end{proposition}

\begin{proof}
Fix $a=\frac{321}{11764}$.  By the Krawczyk inclusion
\eqref{Eqn73}, the reduced system
$\widetilde v_2=\widetilde v_3=\widetilde v_4=0$ has a unique zero
$(\xi_*,\eta_*,g_*)$ in the certified box $\mathcal B_*$.  The exact
reconstruction \eqref{Eqn38} then determines $D_*$ and gives
$v_1(C_5)=0$; hence
\[
 v_1(C_5)=v_2(C_5)=v_3(C_5)=v_4(C_5)=0.
\]
The interval enclosures \eqref{Eqn75} have strictly positive
lower endpoints, so $C_5$ lies in the interior of Hopf Case~$\mathrm{(A)}$.
Moreover, \eqref{Eqn76} separates $v_5(C_5)$ from zero, proving
that $C_5$ is a weak focus of order five.

It remains only to verify the nondegenerate unfolding required by
Corollary~\ref{Coro5.1}. The Schur-complement identity
\eqref{Eqn77}, together with the separation of
$\partial_Dv_1$ from zero and the interval estimate
\eqref{Eqn78}, yields the rank condition
\eqref{Eqn79}. Thus $(v_1,v_2,v_3,v_4)$ provide four independent
local unfolding directions.  The Hopf transversality factor is nonzero by
\eqref{Eqn33}, and strict admissibility persists under sufficiently small
perturbations. Corollary~\ref{Coro5.1} with $k=5$ therefore yields
five distinct small-amplitude limit cycles bifurcating from ${\rm E}_1$.
\end{proof}

\subsubsection{Parameters for independent verification and dimensional 
realization}\label{Sec-5.5.5}

For reproducibility, we first give the corresponding parameter values for the
scaled system~\eqref{Eqn7}.  We use the scaled system because it contains six
parameters, whereas the original system~\eqref{Eqn4} contains eight.  Once a
parameter set for the scaled system is fixed, two parameters in the original
system remain free.  The equivalence of systems~\eqref{Eqn4} and
\eqref{Eqn7} under the transformations~\eqref{Eqn6} and~\eqref{Eqn8}
therefore allows the same certified five-limit-cycle configuration to be
represented by a two-parameter family of dimensional parameter sets.

The six parameters $(p,q,s,\widetilde X,B_{\rm H},r)$ of the scaled system are
\[
\begin{array}{rrr}
p=1.0273089427\ldots, & q=3.2848749860\ldots, &
s=61.7832417793\ldots,\\[1ex]
\widetilde X=0.0065866579\ldots, & B_{\rm H}=0.0000067904\ldots,  
& r=62.2836378231\ldots .
\end{array}
\]
These are the parameters associated with the certified order-five weak focus
obtained above and are not altered in the dimensional reconstruction below.
The focus-value computation was carried out with the Maple procedures of
\cite{yu1998,tianyu2013} at up to $1000$ decimal digits.  For independent
verification, the corresponding high-precision parameter values are recorded
in Appendix~A.

We next transfer this parameter point back to the original dimensional
system~\eqref{Eqn4}.  Introduce the two free positive scaling parameters
\[
\tau=d+\delta>0,
\qquad
\rho=\frac{\Lambda}{d}>0.
\]
Inverting the parameter relations in~\eqref{Eqn8} gives
\begin{equation}\label{Eqn80}
\mu=s\tau,\qquad
d=(r-s)\tau,\qquad
\delta=(1-r+s)\tau.
\end{equation}
Moreover, using the equilibrium relation,
\[
K=\frac{r(1+B_{\rm H})\widetilde X}
{1-(s+1)\widetilde X},
\qquad
k_1=\frac{K}{\widetilde X^p},
\qquad
b=\frac{B_{\rm H}}{\widetilde X^q},
\]
and hence
\begin{equation}\label{Eqn81}
\Lambda=\rho d,\qquad
k=\frac{k_1\tau}{\rho^p},\qquad
\alpha=\frac{b}{\rho^q}.
\end{equation}
Thus, changing $\tau$ and $\rho$ changes only the dimensional realization and
does not alter the scaled parameter point, the focus values, or the
five-limit-cycle bifurcation established above.

To obtain a dimensional realization whose epidemiological time scales are
consistent with the representative ranges summarized in
Table~\ref{table1}, we choose the simple values
\begin{equation}\label{Eqn82}
\tau=0.00275,\qquad \rho=1.
\end{equation}
With this choice, \eqref{Eqn80} gives
$$
\begin{aligned}
\mu =0.1699039149 \cdots,\qquad 
d =0.0013760891\cdots,\qquad 
\delta =0.0013739109\cdots.
\end{aligned}
$$
In particular,
$$
\frac{1}{\mu}=5.8856795655 \cdots\ {\rm days},
\qquad
\frac{1}{\delta}=727.8492475800 \cdots\ {\rm days}
\approx 1.99 \ {\rm years}.
$$
Thus, the infectious period lies within the representative range of
$3$--$14$ days, while the immunity duration is approximately two years and
lies within the representative range of $6$--$24$ months reported in the
literature.  Equivalently, the reconstructed values
\[
\mu=0.1699039149\cdots\ {\rm day}^{-1},
\qquad
\delta=0.0013739109\cdots\ {\rm day}^{-1}
\]
lie within the corresponding ranges summarized in
Table~\ref{table1}.

Since $\rho=1$, we also have
\[
\Lambda=d=0.0013760891\cdots .
\]
Using~\eqref{Eqn81}, the remaining dimensional
parameters are
\[
k=0.3349924727\cdots, \qquad \alpha=0.1463566352\cdots .
\]
Consequently, one dimensional realization of the certified five-limit-cycle
configuration is
\begin{equation}\label{Eqn83}
\begin{array}{llll}
p=1.0273089427\ldots, &
q=3.2848749860\ldots, &
\delta=0.0013739109\ldots, &
\hspace*{-0.30in} k=0.3349924727\ldots,\\[1ex]
\mu=0.1699039149\ldots, &
\alpha=0.1463566352\ldots, &
\Lambda=d=0.0013760891\ldots .
\end{array}
\end{equation}

We emphasize that the choice~\eqref{Eqn83} is not a parameter
fit to a particular disease or data set.  Rather, it is one convenient
dimensional realization, among the two-parameter family generated by
$\tau$ and $\rho$, for which the recovery and loss-of-immunity time scales
are consistent with representative epidemiological ranges.  

\begin{corollary}\label{Coro5.8}
For the dimensional parameter values in~\eqref{Eqn83},
system~\eqref{Eqn4} has a Hopf critical point at which an admissible local
unfolding generates five small-amplitude limit cycles near the positive
equilibrium.  Consequently, the Hopf bifurcation of the original system has
codimension at least five.
\end{corollary}

\begin{proof}
The parameter values in~\eqref{Eqn83} are obtained from
the certified scaled parameter point by the inverse transformations
\eqref{Eqn6}--\eqref{Eqn8}, with $\tau=0.00275$ and $\rho=1$.
Since $s<r<s+1$, all reconstructed demographic rates are positive.
The transformations consist of regular positive scalings of the state
variables and time.  They therefore preserve the local phase-space
correspondence near the positive equilibrium. Hence, the Hopf critical point
and its admissible unfolding generating five small-amplitude limit cycles in
the scaled system correspond directly to those of the original dimensional
system~\eqref{Eqn4}.
\end{proof}

\subsubsection{Biological interpretation}\label{Sec-5.5.6}

The certified order-five weak focus identifies a highly degenerate local
oscillatory regime near the endemic equilibrium.  An ordinary Hopf
bifurcation creates a single nearby periodic branch, whereas the higher-order
unfolding proved above permits five nested small-amplitude periodic solutions
to coexist for nearby admissible parameter values.  In the SIRS variables,
these periodic solutions represent distinct recurrent oscillatory regimes of
the infected and recovered populations around the same endemic state. Thus,
the nonlinear incidence mechanism can support a considerably richer local
oscillation structure than is visible from linear stability or an ordinary
Hopf analysis alone.

The conclusion is local and bifurcation-theoretic.  The five limit cycles are
five nearby periodic orbits in parameter-dependent phase space; they should
not be interpreted as five successive epidemic waves in a single observed
time series.  Rather, their coexistence demonstrates the sensitivity of the
endemic oscillatory dynamics to small admissible changes of epidemiological
parameters near the certified codimension-five point.

\subsection{Analytical and numerical evidence against a codimension-six
Hopf critical point}\label{Sec-5.6}

The codimension-five result of Subsection~\ref{Sec-5.5} is rigorous.
A codimension-six generalized Hopf point would additionally require
\begin{equation}\label{Eqn84}
V_1=V_2=V_3=V_4=V_5=0,\qquad V_6\ne0.
\end{equation}
We do not claim a global nonexistence theorem for \eqref{Eqn84}.
Instead, we record the exact algebraic restrictions and the numerical
searches that motivate the conjecture at the end of this subsection.

\subsubsection{Algebraic necessary conditions}\label{Sec-5.6.1}

After eliminating $D$ by \eqref{Eqn38}, let
$\widehat V_j(a,\xi,\eta,g)$, $j=2,3,4,5$, denote the reduced focus
numerators after factors known to be nonzero in the strict admissible region
have been removed, and define
\begin{equation}\label{Eqn85}
F_{23a}=\operatorname{Res}_{g}(\widehat V_2,\widehat V_3),\quad 
F_{24a}=\operatorname{Res}_{g}(\widehat V_2,\widehat V_4),\quad 
F_{25a}=\operatorname{Res}_{g}(\widehat V_2,\widehat V_5).
\end{equation}
Hence, every strict Case~$\mathrm{(A)}$ solution of \eqref{Eqn84}
must project to a positive solution of
\[
F_{23a}=F_{24a}=F_{25a}=0,
\]
for which the reconstructed parameters satisfy the strict
Case~$\mathrm{(A)}$ admissibility condition $\Omega_1$ defined in
\eqref{Eqn40}. 
This implication is only necessary, since resultant equations may contain
projection artifacts.

There is also an exact restriction on the regular projected branch.  Put
\begin{equation}\label{Eqn86}
R_{24}=\operatorname{pp}_{\eta}
\!\left(\operatorname{prem}_{\eta}(F_{24a},F_{23a})\right),\quad
R_{25}=\operatorname{pp}_{\eta}
\!\left(\operatorname{prem}_{\eta}(F_{25a},F_{23a})\right),
\end{equation}
where $\operatorname{pp}_{\eta}$ denotes the primitive part with respect to
$\eta$.  Exact computation gives
\begin{equation}\label{Eqn87}
\gcd(F_{23a},R_{24})=
\gcd(F_{23a},R_{25})=
\gcd(R_{24},R_{25})=1.
\end{equation}
Thus, the regular projected system has no common hypersurface component.
Indeed, pseudo-division reduces every regular common zero of
$F_{23a},F_{24a},F_{25a}$ to a common zero of
$F_{23a}=R_{24}=R_{25}=0$, and \eqref{Eqn87} excludes a common
nonconstant factor.  This does not exclude isolated or lower-dimensional
real intersections, and content and leading-coefficient branches must be
treated separately.

\subsubsection{Numerical and admissibility evidence}\label{Sec-5.6.2}

We first continued the codimension-five component through the certified point
$C_5$. Writing
\[
H_1(a,\xi,\eta)=F_{23a}(a,\xi,\eta),\quad
H_2(a,\xi,\eta)=F_{24a}(a,\xi,\eta),
\]
the implicit-function theorem locally gives $\xi=\xi(a)$ and
$\eta=\eta(a)$ whenever
$\det\big[\frac{\partial(H_1,H_2)}{\partial(\xi,\eta)}\big]$ $\ne0$. 
Along this branch, we reconstruct $g(a)$ and $D(a)$ and evaluate
\[
 \Phi(a)=F_{25a}(a,\xi(a),\eta(a)).
\]
Predictor--corrector continuation was carried out from $C_5$ in both
directions.  On the computed portion for which
$A_1,A_2,A_3>0$, no zero of $\Phi$ was detected before the branch left the
strict-admissible region.  This excludes no other component and is used only
as numerical evidence.

For a broader projected search, reconstruct $g$ and $D$ at positive roots of
$F_{23a}=F_{24a}=0$ and substitute these rational expressions into the three
Case~$\mathrm{(A)}$ margins $A_1,A_2,A_3$.  On each branch where the relevant
denominator signs are fixed, clear the known positive denominators in the
pairwise necessary admissibility inequalities.  We denote the resulting
polynomial numerators by
$\mathcal N_{12},\mathcal N_{13},\mathcal N_{23}$.  Their expanded forms are
not needed here; only their signs are used.  By construction, strict
Case~$\mathrm{(A)}$ admissibility implies
\begin{equation}\label{Eqn88}
\mathcal N_{12}>0,\quad
\mathcal N_{13}>0,\quad
\mathcal N_{23}>0.
\end{equation}
Consequently, failure of any inequality in \eqref{Eqn88}
excludes that projected root from the strict Case~$\mathrm{(A)}$
codimension-six set.

The global component scan used $30$ logarithmically spaced slices over
$10^{-5}\le a\le10^2$. It produced $171$ distinct positive projected roots
of $F_{23a}=F_{24a}=0$, and none satisfied all three necessary inequalities
in \eqref{Eqn88}. Thus, all $171$ sampled roots were excluded
from strict Case~$\mathrm{(A)}$. The scan is finite and non-exhaustive, and
is therefore recorded only as computational evidence.

\subsubsection{Codimension-six conjecture}\label{Sec-5.6.3}

Combining the exact and numerical information, every strict 
Case~$\mathrm{(A)}$ codimension-six point must satisfy
\[
F_{23a}=F_{24a}=F_{25a}=0,\quad \Omega_1>0,
\]
while the regular projected equations satisfy the coprimality relations
\eqref{Eqn87}. No codimension-six point was found either on the
continued component through $C_5$ or in the global projected scan described
above.  Neither computation constitutes an exhaustive global exclusion.  In
particular, no exact exclusion has been established on every regular,
content, and leading-coefficient branch, and the computation above concerns
strict Case~$\mathrm{(A)}$; Case~$\mathrm{(B)}$ remains open.

\begin{conjecture}\label{Conj5.9}
There is no strict admissible generalized-Hopf point satisfying
\[
 V_1=V_2=V_3=V_4=V_5=0.
\]
Consequently, five is the maximal number of small-amplitude limit cycles
generated from the positive equilibrium by a generalized Hopf bifurcation in
the present SIRS family.
\end{conjecture}

\section{Comparison with \(p=7,\,q=5\) for two small-amplitude limit cycles}
\label{Sec-6}

For comparison with the five-cycle result of the previous section, 
we consider the distinct parameter choice \(p=7\) and \(q=5\), 
for which the Hopf bifurcation can be analyzed more explicitly.

\begin{proposition}\label{Prop6.1}
For system \eqref{Eqn13}, when \(p=7\) and \(q=5\), 
two small-amplitude limit cycles can bifurcate from the Hopf critical point
\[
B=B_{\mathrm H}=\frac{6-a}{a-1}
\]
near the positive equilibrium \({\rm E}_1=(1,1)\).
\end{proposition}

\begin{proof}

Using \eqref{Eqn30} and setting \(p=7\) and \(q=5\), we obtain
\[
c=6-a,\qquad h=a-1,
\]
which requires \(1<a<6\). Thus, \(a\), \(D\), and \(g\) remain as the 
three free parameters in the computation of the focus values, under 
the following restriction (see the conditions in \eqref{Eqn32}):
\begin{equation}\label{Eqn89} 
D>0, \ \ g>0, \ \  
\left\{
\begin{array}{ll} 
\textrm(A) \ \ \max\left\{1,\,\dfrac{g(g+1)}{1+D+g},\, 
\dfrac{6(g+1)}{1+D+g} \right\} 
<a < \min\left\{6,\,\dfrac{g(g+1)}{1+D} \right\}, 
\\[3.5ex] 
\textrm(B) \ \ \max\left\{1,\,\dfrac{g(g+1)}{1+D+g} \right\} 
<a < \min\left\{6,\,\dfrac{g(g+1)}{1+D},\, \dfrac{6(g+1)}{1+D+g} \right\}. 
\end{array}
\right.  
\end{equation} 

We first assume only $1<a<6$ and $g>0$ and prove that $v_2\ne0$ 
whenever $v_1=0$ and $D>0$. We then choose admissible parameter 
values satisfying $v_1=0$. 

In particular, from \eqref{Eqn36}, we obtain
\[
D=-\,\frac{(g+1)(4a-3)\big[(a-1)(6-a)+2(4a-3)g\big]}
{aF_1(a,g)},
\]
where
\[
F_1(a,g)=(a-1)(a-4)(a-6)-(8a^2-31a+24)g.
\]
Since the numerator in the expression for \(D\) is positive 
for \(1<a<6\) and \(g>0\), we have \(D>0\) if and only if \(F_1<0\).

The second focus value can be written as
\[
v_2=
\frac{-5\,v_{2c}}
{48a(a-1)(4a-3)\big[(a-1)(6-a)+2(4a-3)g\big]},
\]
where
$$
\begin{array}{rl}
v_{2c}= \!\!\! & (-672a^5+3420a^4-6414a^3+8223a^2-7836a+3276)g^2\\[1.0ex] 
&+(a-1)(a-6)(28a^4-85a^3+369a^2-1366a+1092)g\\[1.0ex]
&+(a-1)^2(a-6)^2(7a^3-45a^2-10a+91).\\[1.0ex] 
\end{array} 
$$
Therefore, for \(1<a<6\) and \(g>0\), it is sufficient to show 
that \(F_1<0\) and \(v_{2c}\ne0\).

Define
\[
Q(a)=8a^2-31a+24=8(a-\alpha)(a-\beta),
\]
where
\[
\alpha=\frac{31-\sqrt{193}}{16},
\qquad
\beta=\frac{31+\sqrt{193}}{16}.
\]

There are three possible intervals:
\[
1<a<\alpha,\qquad
\alpha\le a\le\beta,\qquad
\beta<a<6.
\]
However, the middle interval can be excluded immediately. 
Indeed, for \(\alpha\le a\le\beta\), we have \(Q(a)\le0\), while
\((a-1)(a-4)(a-6)>0\), since \(1<\alpha<\beta<4\). 
Hence, for \(g>0\), \(F_1(a,g)>0\). Thus \(F_1<0\), 
and consequently \(D>0\), is impossible in this interval. 
We therefore only need to consider the following two cases.
Note that the $\alpha$ used in this section should not be confused with 
the $\alpha$ used in model \eqref{Eqn3}. 

\medskip

\noindent
\textbf{Case 1:} \(1<a<\alpha\).

In this case, \((a-1)(a-4)(a-6)>0\) and \(Q(a)>0\). 
Hence, \(F_1(a,g)<0\) implies
\[
g>G(a)
=
\frac{(a-1)(a-4)(a-6)}
{8a^2-31a+24}.
\]
Write \(g=G(a)+t\), where \(t>0\). Substituting this expression into \(v_{2c}\) and simplifying gives
\[
v_{2c}=A(a)t^2+E(a)t+H(a),
\]
where
\[
\begin{array}{ll}
A(a)=-3\left(224a^5-1140a^4+2138a^3-2741a^2+2612a-1092\right),\\[1ex]
E(a)=-\dfrac{7a(a-6)(a-1)R(a)}
{8a^2-31a+24},\\[2ex]
H(a)=-\dfrac{49a^2(a-6)^2(a-1)^2(4a-3)(5a^2-20a+17)}
{(8a^2-31a+24)^2},
\end{array}
\]
with
\[
R(a)=160a^5-1524a^4+4847a^3-6193a^2+3074a-372.
\]

Since \(\alpha<\frac{27}{25}\), it suffices to consider \(1<a<\frac{27}{25}\). 
By Sturm's theorem, \(A\) has no zero on \([1,\frac{27}{25}]\). 
Since \(A(1)=-3<0\), it follows that \(A(a)<0\) throughout this interval. 
Similarly, Sturm's theorem shows that \(R\) has no zero 
on \([1,\frac{27}{25}]\), and since \(R(1)=-8<0\), we have \(R(a)<0\) there.

Consequently, \(E(a)<0\). Moreover, \(4a-3>0\) 
and \(5a^2-20a+17=5(a-2)^2-3>0\) for \(1<a<\frac{27}{25}\), 
so \(H(a)<0\). Therefore, since \(t>0\),
\[
v_{2c}=A(a)t^2+E(a)t+H(a)<0.
\]
Thus \(v_{2c}\ne0\) in Case 1.

\medskip

\noindent
\textbf{Case 2:} \(\beta<a<6\).

In this case, regard \(v_{2c}\) as a quadratic polynomial in \(g\):
\[
v_{2c}=A_2(a)g^2+B_2(a)g+C_2(a).
\]
Its discriminant with respect to \(g\) is
\[
\Delta_g(v_{2c})
=
49a^2(a-6)^2(a-1)^2P(a),
\]
where
\[
P(a)=400a^6-4520a^5+16249a^4-23318a^3+12809a^2-660a-920.
\]

Since \(\beta>\frac{14}{5}\), it suffices to prove that 
\(P(a)<0\) on \([\frac{14}{5},6]\). 
By Sturm's theorem, \(P\) has no zero on this interval. Since
\[
P\left(\frac{14}{5}\right)=-\frac{387596}{625}<0,
\]
we conclude that \(P(a)<0\) for \(\frac{14}{5}\le a\le6\). Hence
\[
\Delta_g(v_{2c})<0
\qquad (\beta<a<6).
\]
Therefore, \(v_{2c}\), viewed as a quadratic polynomial in \(g\), 
has no real zeros. In particular, \(v_{2c}\ne0\) 
for \(\beta<a<6\) and \(g>0\).

Combining the two cases, we conclude that
\[
F_1(a,g)<0
\quad\Longrightarrow\quad
v_{2c}(a,g)\ne0
\]
for \(1<a<6\) and \(g>0\). 

Having proved that for $p=7$ and $q=5$, the codimension of 
the Hopf bifurcation is at most two (Bautin bifurcation). 
Now, we want to verify if there exist solutions such that all the 
constraints on the parameters, in particular,
$a_{\min}<a<a_{\max}$, are satisfied. 
This is a simple verification, since we only need to 
find a solution satisfying all the conditions. For example, 
choosing $a=5$ and $g=10$ gives $D=\frac{32164}{1735}$, 
and obtain that the condition (A) is satisfied: 
$$ 
a-\dfrac{g(g+1)}{1+D+g}= \dfrac{5945}{4659}, \quad  
a-\dfrac{6(g+1)}{1+D+g}= \dfrac{4295}{1553}, \quad   
\dfrac{g (g+1)}{1+D}-a=\dfrac{21355}{33899},  
$$ 
all of them are positive, 
indicating that the required conditions are satisfied.  

This completes the proof.
\end{proof}
This example also highlights the role of allowing \(p\) and \(q\) to vary
over the real numbers. For the fixed integer pair \(p=7,\ q=5\), the
degeneracy obtained here yields only two small-amplitude limit cycles, 
whereas the free-exponent construction in Section~5 reaches a nondegenerate 
weak focus of order five. Thus, the additional freedom in the real 
exponents is not merely a technical convenience: it permits substantially 
higher Hopf degeneracy.

\section{A simple proof for the codimension-$2$ 
BT bifurcation of system \eqref{Eqn13}}\label{Sec-7}

In this section, as a showcase, we apply our method to prove that 
the codimension of the BT bifurcation of system \eqref{Eqn13} is two, 
providing a comparison with the approach presented in~\cite{cui2024saddle}. 

We do not provide a detailed bifurcation analysis for the BT bifurcation, 
since it has already been analyzed in detail in~\cite{cui2024saddle}. Instead, 
we use this example to demonstrate the advantages of our approach.

\begin{proposition}\label{Prop7.1}
For system \eqref{Eqn13}, the BT bifurcation occurs at the 
positive equilibrium ${\rm E_1}=(1,1)$, and its codimension is two.
\end{proposition} 

\begin{proof}
The BT bifurcation point is determined when the trace ${\rm Tr}$ 
and determinant $\det$ of the Jacobian matrix of system \eqref{Eqn13} 
vanish simultaneously.   
The restriction on the parameters is given by \eqref{Eqn16}: 
\begin{equation}\label{Eqn90} 
p>0,\quad q\ge 0,\quad \tilde{X}>0, \quad B \ge 0, \quad 
\textrm{and} \quad 0<s<r<s+1.  
\end{equation}   

Evaluating the determinant of the Jacobian of \eqref{Eqn13} 
at the equilibrium ${\rm E_1}=(1,1)$ gives 
$$ 
\det = \frac{r}{(1+B) \tilde{X}_1} \big[1+B+B q \tilde{X}_1 
-p (B+1)\tilde{X}_1 \big].
$$ 
Setting $\det = 0$ yields 
\begin{equation}\label{Eqn91} 
p_{\rm bt} = \dfrac{1+B+ B q\tilde{X}_1}{(1+B) \tilde{X}_1}. 
\end{equation} 
Under this condition, the trace of he Jacobian of \eqref{Eqn13} becomes
$$ 
{\rm Tr} = \dfrac{r s \tilde{X} - \tilde{X}_1}{\tilde{X}_1}.  
$$  
Letting ${\rm Tr} =0$ gives 
\begin{equation}\label{Eqn92} 
r_{\rm bt} = \dfrac{\tilde{X}_1}{s\,\tilde{X}}, 
\end{equation}
which requires 
\begin{equation}\label{Eqn93} 
0<r_{\rm bt}-s = \dfrac{ 1-(s^2+s+1) \tilde{X}}{s \tilde{X}}<1 
\ \ \Longrightarrow \ \ 
\dfrac{1}{(s+1)^2} < \tilde{X} < \dfrac{1}{s^2+s+1}. 
\end{equation}
The positive equilibrium ${\rm E_1}$ is asymptotically stable if 
$p<p_{\rm bt}$ and $r < r_{\rm b}$.  

Moreover, a direct calculation shows that 
\begin{equation}\label{Eqn94} 
\det\left[\begin{array}{cc} \dfrac{\partial {\rm Tr}}{\partial p} & 
\dfrac{\partial {\rm Tr}}{\partial r} \\[2.0ex]
\dfrac{\partial \det}{\partial p} &    
\dfrac{\partial \det}{\partial r} 
\end{array} \right]_{(p=p_{\rm bt},\, r=r_{\rm bt})}
= 1, 
\end{equation}
implying that the unfolding of $p$ and $r$ is non-degenerate.

Introducing the following affine transformation,
$$ 
\left(\begin{array}{c} x \\ y \end{array} \right) 
= \left(\begin{array}{c} 1 \\ 1 \end{array} \right) 
+ \left[\begin{array}{cc} 1 & 1 \\[1.0ex] 
                          1 & 0 
        \end{array} 
  \right] 
\left(\begin{array}{c} u \\ v \end{array} \right),
$$ 
into system \eqref{Eqn13}, together with the above expressions for 
$p$ and $r$, we obtain 
$$ 
\hspace*{-0.20in}
\begin{array}{rl} 
\dfrac{d u}{dt} = \!\!\!\! & v, \\[2.0ex] 
\dfrac{d v}{dt} = \!\!\!\! &
\dfrac{1}{s \tilde{X} [1+B (1+u+v)^q]} 
\Big\{ (1+B) \big[1- \tilde{X} \big( (1+s) (1+u)+v \big) \big]
        (1+u+v)^{\frac{1+ B + B q \tilde{X}_1}{(1+B) \tilde{X}_1}} 
\\[2.0ex] 
& \hspace*{1.5in} 
-\big[1+B (1+u+v)^q \big]
\big[1+u+v- \tilde{X} \big((1+s)(1+u)+v \big) \big] \Big\} 
\\[1.0ex] 
:= \!\!\!\! & f_2(u,v). 
\end{array} 
$$  
Thus, the normal form for the BT bifurcation up to second order is 
\begin{equation}\label{Eqn95} 
\begin{array}{rl} 
\dfrac{d y_1}{dt} = \!\!\! & y_2, \\[2.0ex] 
\dfrac{d y_2}{dt} = \!\!\! & c_{20}\, y_1^2 + c_{11} \, y_1 y_2,  
\end{array} 
\end{equation} 
where 
\begin{equation}\label{Eqn96} 
c_{20} = \dfrac{1}{2} \ \dfrac{\partial^2 f_2}{\partial u^2} 
=-\,\dfrac{c_{20a}}{2 s \tilde{X} (1+B)^2 \tilde{X}_1},
\qquad 
c_{11} = \dfrac{\partial^2 f_2}{\partial u v} 
= -\,\dfrac{c_{11a}}{s \tilde{X} (1+B)^2 \tilde{X}_1},
\end{equation} 
in which 
$$
\begin{array}{ll}
c_{20a}= B q^2 (s+1)^2 \tilde{X}^2+(s+1) \big[(B+1)^2-2 B q^2 \big] \tilde{X}
+B q^2, \\[1.0ex] 
c_{11a}= B q^2 (s+1)^2 \tilde{X}^2+\big[(B+1)^2-2 B (s+1) q^2\big] \tilde{X}
+B q^2,
\end{array} 
$$
both of them are quadratic functions in $\tilde{X}$. 
The discriminant of $c_{20a}$ is 
$$ 
\Delta = (s+1)^2 (1+B)^2 \big[(B+1)^2-4 B q^2 \big] < 0 \quad 
\ \ \textrm{for} \ \ q^2 > \dfrac{(B+1)^2}{4B} \ \ 
\Longrightarrow \ \ c_{20a}>0. 
$$ 
When $q^2 \le \frac{(B+1)^2}{4B}$, we have 
$$ 
(B+1)^2-2 B q^2 \ge \dfrac{(B+1)^2}{2} > 0 \quad \Longrightarrow \ \ 
c_{20a}>0.  
$$ 
Therefore, under the conditions \eqref{Eqn90} and \eqref{Eqn93}, 
$c_{20a}>0$ and thus $c_{20}<0$. Similarly, we can show that 
$c_{11}<0$. 
This indicates that the BT bifurcation is at most of condimension two. 
Furthermore, with the non-degenerate condition \eqref{Eqn94}, 
we can conclude that the codimension of the BT bifurcation is two. 
Since our objective here is to demonstrate our method for determining the 
codimension, we refer the reader to the publication~\cite{cui2024saddle} 
for the detailed bifurcation analysis.
\end{proof}

\section{Conclusion and discussion}\label{Sec-8}

The principal result of this paper is a rigorous lower bound on the Hopf
codimension for the SIRS model with generalized nonlinear incidence. Using a
nondimensionalization and parametrization adapted to symbolic focus-value
calculations, the higher focus values are expressed directly in the five
parameters $(a,m,n,g,D)$, with $m=p$ and $n=q$. The resulting Hopf
conditions are algebraic and the biological restrictions become explicit
semialgebraic admissibility inequalities.

Under these admissibility conditions, we rigorously prove the existence of
strict-interior nondegenerate weak foci of orders three, four, and five.  In
the codimension-five case, the first four focus values are certified to
vanish at an isolated admissible point, the fifth focus value is rigorously
separated from zero, and the unfolding Jacobian is nonsingular. The local
normal-form argument then proves the existence of five distinct
small-amplitude limit cycles bifurcating from a single positive equilibrium.
Thus,
\begin{equation}\label{Eqn97}
\operatorname{codim}(\mathrm{BT})\le 2,
\qquad
\operatorname{codim}(\mathrm{Hopf})\ge 5.
\end{equation}
The gap in \eqref{Eqn97} is a notable feature of this model.

The codimension-six question has a different logical status.  Exact
elimination and coprimality calculations, continuation of the certified
codimension-five component, and a global projected Case~\(\mathrm{A}\)
search produce no admissible codimension-six candidate. These calculations
constitute analytical and numerical evidence, not a complete real-algebraic
nonexistence proof. Accordingly, maximality of five cycles is formulated as
a conjecture rather than a theorem.

The analysis therefore establishes rigorously that the generalized incidence
exponents permit Hopf degeneracy of codimension at least five and that five
small-amplitude epidemic oscillations can arise from a single Hopf critical 
point. A natural remaining problem is to determine whether the codimension-six
obstruction can be strengthened to an exact semialgebraic nonexistence
result, thereby deciding whether five is the true maximal number of
small-amplitude limit cycles for this model.

\section*{Acknowledgment}

This work was supported by the Natural Sciences and Engineering 
Research Council of Canada, Grant No.~R2686A07 (P. Yu).
Wanyue Tang is grateful to the China Scholarship Council for funding 
her visit to Western University.

\section*{Appendix A}

The values of the parameters $p$, $q$, $s$, $\tilde{X}$, $B=B_{\rm H}$, 
and $r$, given up to $1000$ decimal points are listed below. 
{\footnotesize 
$$
\!\!\! 
\begin{array}{rl}
p=\!\!\!\!\!& 1.0273089427030402844604318458355364324043191900641683223731696163095337763828821082133683404
\\
&423828802625349260254849646993618204015595138663617783423114205711045535908146813096142994037
\\
&666134721532014278385345606709520244857462299532969410974794034011768131005356808326174621728
\\
&017618620593504430058451609180860863149284366824225722056415654842370228562668001322025262199
\\
&312791636012122905916796587902993625743709393221412178986995102668991165951851848110168707219
\\
&943705766175218144533701829916003213014068462120471884418361593537028815257672619585215078234
\\
&517828952252806754403796249824473052645921920095965052801114634103880556843043956384618389292
\\
&854593504867221527829276280565065493286207778070228018946510171682495621465582507375636955068
\\
&270124404447892461728143041147424333855610068449591672708329455202597669463305684017629219610
\\
&495058180593857487404442892343230344566298950761439761658834413680839479743124801225968649350
\\
&43193179771123705997479211283356774153244700531378301006053102640664493, 
\\
q=\!\!\!\!\!&3.2848749860002783384877996679273324202506499390916818917882588683177056902450443073697212500
\\
&648020947419770343948954235674462638730199316155015819423485690394245802167285603663438278662
\\
&933366233197482805941915862998252257247601909397145168939531041530142803076371947463628280798
\\
&349731397665023847671540267847914747350236021139983178781471157815731068602191964169782756377
\\
&869794287714164245569272471260286269473829943590151560079655508312145165689316298730330983961
\\
&115803359847165516786410048099192079394890893647502803199809152930057878053738616588507177236
\\
&468101966276237235627728669488909854063703877592888818974215554569554213498573312765211253855
\\
&992266633146988403125191747671194640296809949329293402160167013190956792842156455615357234371
\\
&439541963614501968876701591836489762371207249938348990312162438242730294193413339178163464609
\\
&462475697793889472212082394807700016895222320370465581812669448656297345586544033483091123659
\\
&07362305751105346778773445694152713814855919280132067160540259298946719,
\\
\end{array} 
$$
}
{\footnotesize  
$$
\!\!\! 
\begin{array}{rl}
r=\!\!\!\!\!&62.283637823145089533998698955674427050096216101177610973978236196496214513182198570812319057
\\
&472468133042285864158353391488838676394685130594428222432819173903833412747586730588618011221
\\
&442788614066917133694648717957228388344957449724866488004974494469330163883049676372431970901
\\
&363105402871279055085683650642528048905116740941951487024308471546414507408554285382973250405
\\
&637843096739145300242490482377899527769107654966083269563377267946491753658631247225011621425
\\
&193314668535005460030021120766311132926351121829893160381629078904240273464754922967068827628
\\
&704938651810048106533847939854224912397980577552109971596563838201732900797654834992650513780
\\
&956517405263670159074387530347612638687609045217074432895295726870332353853095308823162416842
\\
&185465885704704273076965754108543827926015680671502638461768735119373615316314100430175148257
\\
&770479284626721759193157197167307048151666905606542163832697274710774210343497170319811706886
\\
&03023427789829457452519083788590704180077018159835965307748286241385397,
\\ 
s=\!\!\!\!\!&61.783241779383772960060785232440112174960768263739968974518165388666331132975947946170082036
\\
&424864598692582645055155399414620860587305660710866552049423644737547578423017138388288331366
\\
&679551513786723758953228655008153248227571151550161303304072006962850431596941788332708040413
\\
&390882210774612204971764616302491780220621386600047336537569355789292208671780581258912498730
\\
&811994994334334508103678115596492641658844913386575784021510347794961498603362344190682404670
\\
&299039393418992616958615877562528077732282590428834220022384763852687170770819309150888694289
\\
&404533726179084038389729842052542063112113664512441722455783543726530080854012336724622971673
\\
&259240043364291114437928841872231600274005161108887907813573582430920615926627549446750466880
\\
&351710603488952143042786726515122012286049090368730098865776184183984524504644603470059424555
\\
&140815039790644927208579302839358739520349322566751160322826737726101579097932684851409220173
\\
&34990582326028046383715494296997240123983528343933737745129166329368763,
\\
\tilde{X}=\!\!\!\!\!&0.0065866579367628975052581424277616444651207611620331725801628217739189325101184947980003641
\\
&781529374710696363123251051023602853351921247633309811022327393887163404075748239949323856856
\\
&648763090555986699718959288244041152212633759857669820060051212524094885679802901146616659320
\\
&631910708241677988921643527898596907922529225400309469820469779162517400587432818603705154652
\\
&359560386962486601016126784312489189885753747455229321714095444126038217751230272587130097912
\\
&131313406913532028077255443827451979039503035207263750187829204915482329111520585581668144248
\\
&325412736108949893577808976506398317016965646660284712395600491107816700976530136294306717601
\\
&872814716972374520934075732967688588914876960827165197859040294259211941454172696956795559857
\\
&802775155378953364637497028135823450971884559798167927931253106879440922374399641249771639284
\\
&871385976193959226805310039180548424960040734900831276174868317807910699584515739294350161855
\\
&94404490895190090001092543914659258099641565118073206194563696765563574213,
\\ 
B=\!\!\!\!\!&0.0000067904125322984352081055607186420657998404544751514266032113084931595475180858189621688
\\
&408448194372058290905791613822748246764852975548043299932013017376386349490584196261716543433
\\
&206596332499718790642091339583227448400323706916107173212749242003633908453063074842257585758
\\
&471838002740663970572238854408698212110954957706031996586672366892930824805107967466641703459
\\
&648613406679484424989806648359532687067842299293719211634897474790044377677736641128592170064
\\
&921136404759727495798562986295399819530775614539779392333868615288363514974439223508977454412
\\
&624247446982237200294025184238297994322116044481102547736821632452584864879700847790309521756
\\
&428289627659063833703772895676145206246921972922855792905714036916072759608002018645253666602
\\
&326919850163957195560957551737618036061472158760084067430803390087905904880592172119251542468
\\
&678898724358163197189043339254889436530835414520366245887252311134305399864859460820630586611
\\
&04995707292089830453000884336920616958351398771892014416367786141361554412185.
\end{array}
$$ 
}


\begin{thebibliography}{99}
\bibitem{anderson1979population}
R.~M. Anderson, R.~M. May,
Population biology of infectious diseases: Part I,
Nature 280(5721) (1979) 361--367.

\bibitem{anderson1991infectious}
R.~M. Anderson, R.~M. May,
Infectious Diseases of Humans: Dynamics and Control,
Oxford University Press, Oxford, 1991.

\bibitem{capasso1978generalization}
V. Capasso, G. Serio,
A generalization of the Kermack-McKendrick deterministic epidemic model,
Math. Biosci. 42(1-2) (1978) 43--61.

\bibitem{cui2024saddle}
W. Cui, Y. Zhao,
Saddle-node bifurcation and {Bogdanov--Takens} bifurcation of a 
SIRS epidemic model with nonlinear incidence rate,
J. Differ. Equ. 384 (2024), 252--278.

\bibitem{diekmann2000mathematical}
O. Diekmann, J.~A.~p. Heesterbeek,
Mathematical Epidemiology of Infectious Diseases: 
Model Building, Analysis and Interpretation,
John Wiley and Sons, Chichester, 2000.

\bibitem{hethcote1976qualitative}
H.~W. Hethcote,
Qualitative analyses of communicable disease models,
Math. Biosci. 28(3-4) (1976) 335--356.

\bibitem{hethcote2000mathematics}
H.~W. Hethcote,
The mathematics of infectious diseases,
SIAM Rev. 42(4) (2000) 599--653.

\bibitem{hu2011bifurcations}
Z. Hu, P. Bi, W. Ma, S. Ruan,
Bifurcations of an SIRS epidemic model with nonlinear incidence rate,
Discrete Contin. Dyn. Syst. B 15(1) (2011) 93--112.

\bibitem{hu2012}
Z. Hu, W. Ma, S. Ruan,
Analysis of SIR epidemic models with nonlinear incidence rate and treatment,
Math. Biosci. 238(1) (2012) 12--20.

\bibitem{kermack1927contribution}
W.~O. Kermack, A.~G. McKendrick,
A contribution to the mathematical theory of epidemics,
Proceedings of the Royal Society of London. Series A, 
Containing Papers of a Mathematical and Physical Character,
115 (772) (1927) 700--721.

\bibitem{li2018}
J. Li, Z. Teng,
Bifurcations of an SIRS model with generalized non-monotone incidence rate,
Adv. Differ. Equ. 2018(1) (2018) 1--21.
\bibitem{el2023extending}
M.~El Khalifi and T.~Britton,
Extending susceptible-infectious-recovered-susceptible epidemics to allow for gradual waning of immunity, J. R. Soc. Interface {20}(206) (2023) 20230042.
\bibitem{lizana1996multiparametric}
M. Lizana, J. Rivero,
Multiparametric bifurcations for a model in epidemiology,
J. Math. Biol. 35(1) (1996) 21--36.

\bibitem{liu1986influence}
W.-M. Liu, S.~A. Levin, Y. Iwasa,
Influence of nonlinear incidence rates upon the behavior of SIRS 
epidemiological models,
J. Math. Biol. 23(2) (1986) 187--204.

\bibitem{lu2019bifurcation}
M. Lu, J. Huang, S. Ruan, P. Yu,
Bifurcation analysis of an SIRS epidemic model with a generalized 
nonmonotone and saturated incidence rate,
J. Differ. Equ. 267(3) (2019) 1859--1898.

\bibitem{lu2021global}
M. Lu, J. Huang, S. Ruan, P. Yu,
Global dynamics of a susceptible-infectious-recovered epidemic 
model with a generalized nonmonotone incidence rate,
J. Dyn. Differ. Equ. 33(4) (2021) 1625--1661.

\bibitem{ruan2003dynamical}
S. Ruan, W. Wang,
Dynamical behavior of an epidemic model with a nonlinear incidence rate,
J. Differ. Equ. 188(1) (2003) 135--163.

\bibitem{tang2008coexistence}
Y. Tang, D. Huang, S. Ruan, W. Zhang,
Coexistence of limit cycles and homoclinic loops in a SIRS model 
with a nonlinear incidence rate,
SIAM J. Appl. Math. 69(2) (2008) 621--639.

\bibitem{tianyu2013}
Y. Tian, P. Yu,
An explicit recursive formula for computing the normal form and 
center manifold of general $n$-dimensional differential systems 
associated with Hopf bifurcation,
Int. J. Bifurcation and Chaos 23(6) (2013) 1350104.

\bibitem{tianyu2015}
Y. Tian, P. Yu,
Center conditions in a switching {Bautin} system,
J. Differ. Equ. 259(3) (2015) 1203--1226.

\bibitem{xiao2007global}
D. Xiao, S. Ruan,
Global analysis of an epidemic model with nonmonotone incidence rate,
Math. Biosci. 208(2) (2007) 419--429.

\bibitem{xiao2006qualitative}
D. Xiao, Y. Zhou,
Qualitative analysis of an epidemic model,
Canadian Appl. Math. Quarterly 14(4) (2006) 469--492.

\bibitem{yu1998}
P. Yu,
Computation of normal forms via a perturbation technique,
J. Sound Vib. 211(1) (1998) 19--38.

\bibitem{zhang2022bifurcations}
F. Zhang, W. Cui, Y. Dai, Y. Zhao,
Bifurcations of an SIRS epidemic model with a general saturated incidence rate,
Math. Biosci. Eng. 19(11) (2022) 10710--10730.

\bibitem{zhou2007bifurcations}
Y. Zhou, D. Xiao, Y. Li,
Bifurcations of an epidemic model with non-monotonic incidence rate 
of saturated mass action,
Chaos, Solitons and Fract. 32(5) (2007) 1903--1915.

\bibitem{krawczyk1969newton}
R. Krawczyk,
Newton-Algorithmen zur Bestimmung von Nullstellen mit Fehlerschranken,
Computing 4(3) (1969) 187--201.


\bibitem{moore2009introduction}
R.~E. Moore, R.~B. Kearfott, and M.~J. Cloud,
\newblock {\em Introduction to Interval Analysis},
\newblock SIAM, Philadelphia, 2009.

\bibitem{rump2010verification}
S.~M. Rump,
\newblock Verification methods: rigorous results using floating-point arithmetic,
\newblock {\em Acta Numerica} 19 (2010) 287--449.


\end{thebibliography}

\end{document}